\documentclass[11pt]{amsart}
\usepackage{amsmath, amssymb, latexsym}
\usepackage{mathrsfs}
\usepackage{enumerate}
\usepackage[all, knot]{xy}
\usepackage{color}

\newtheorem{thm}{Theorem}[section]
\newtheorem{prop}[thm]{Proposition}
\newtheorem{cor}[thm]{Corollary}
\newtheorem{lemma}[thm]{Lemma}

\theoremstyle{definition}
\newtheorem{defn}[thm]{Definition}
\newtheorem{Ex}[thm]{Example}

\theoremstyle{remark}
\newtheorem{Rmk}[thm]{Remark}

\newenvironment{red}
{\relax\color{red}}
{\hspace*{.5ex}\relax}

\newcommand{\ber}{\begin{red}}
\newcommand{\er}{\end{red}}

\newenvironment{verd}
{\relax\color{magenta}}
{\hspace*{.5ex}\relax}

\newcommand{\bg}{\begin{verd}}
\newcommand{\eg}{\end{verd}}

\numberwithin{equation}{subsection}

\newcommand{\Z}{\mathbb{Z}}
\newcommand{\Q}{\mathbb{Q}}
\newcommand{\C}{\mathbb{C}}
\newcommand{\A}{\mathbb{A}}

\newcommand{\g}{\mathfrak{g}}

\newcommand{\Hom}{\mathrm{Hom}}
\newcommand{\End}{\mathrm{End}}

\newcommand{\wt}{{\rm wt}}

\newcommand{\Span}{{\rm Span}}

\newcommand{\Top}{{\rm Top}}
\newcommand{\Soc}{{\rm Soc}}
\newcommand{\Rad}{{\rm Rad}}
\newcommand{\Mat}{{\rm Mat}}

\newcommand{\proj}{\mathrm{proj}}

\newcommand{\Ker}{\mathrm{Ker}}
\newcommand{\Tr}{\mathrm{Tr}}

\newcommand{\rlQ}{\mathsf{Q}}   
\newcommand{\wlP}{\mathsf{P}}   
\newcommand{\weyl}{\mathsf{W}}  
\newcommand{\cmA}{\mathsf{A}}  

\newcommand{\ST}{\mathsf{ST}}   
\newcommand{\res}{\mathrm{res}}   
\newcommand{\F}{\mathcal{F}}   

\newcommand{\bR}{\mathbf{k}}   
\newcommand{\sB}{\mathbb{B}}   
\newcommand{\fqH}{R^{\Lambda_0}}   

\renewcommand{\Im}{\operatorname{Im}}
\renewcommand{\Ker}{\operatorname{Ker}}

\newcommand{\sBox}[1]
{
\xy
(-2,-2)*{};(2,-2)*{} **\dir{-};
(-2,2)*{};(2,2)*{} **\dir{-};
(-2,-2)*{};(-2,2)*{} **\dir{-};
(2,-2)*{};(2,2)*{} **\dir{-};
(0,0)*{#1};
\endxy
}

\newcommand{\ssBox}[1]
{
\xy
(-1.5,-1.5)*{};(1.5,-1.5)*{} **\dir{-};
(-1.5,1.5)*{};(1.5,1.5)*{} **\dir{-};
(-1.5,-1.5)*{};(-1.5,1.5)*{} **\dir{-};
(1.5,-1.5)*{};(1.5,1.5)*{} **\dir{-};
(0,0)*{ _{#1}};
\endxy
}

\begin{document}

\title[Representation type of finite quiver Hecke algebras of type $D^{(2)}_{\ell+1}$]
{Representation type of finite quiver \\
Hecke algebras of type $D^{(2)}_{\ell+1}$}

\author[Susumu Ariki]{Susumu Ariki $^1$}
\thanks{$^1$ S.A. is supported in part by JSPS, Grant-in-Aid for Scientific Research (B) 23340006.}
\address{Department of Pure and Applied Mathematics, Graduate School of Information
Science and Technology, Osaka University, Toyonaka, Osaka 560-0043, Japan}
\email{ariki@ist.osaka-u.ac.jp}

\author[Euiyong Park]{Euiyong Park } 
\address{Department of Mathematics, University of Seoul, Seoul 130-743, Korea}
\email{epark@uos.ac.kr}


\begin{abstract}
We give Erdmann-Nakano type theorem for the finite quiver Hecke algebras $\fqH(\beta)$
of affine type $D^{(2)}_{\ell+1}$, which tells their representation type. If $\fqH(\beta)$
is not of wild representation type, we may compute its stable Auslander-Reiten quiver.
\end{abstract}

\maketitle


\vskip 2em

\section*{Introduction}

This paper is the second of our series of papers on the representation type of finite quiver Hecke algebras.
Categorification is popular in various fields of mathematics recent days. One purpose of categorification is for deeper
understanding of the objects to be categorified. But there is another purpose: categorification is a method to relate
various categories in a systematic way. Categorification of integrable modules over a Kac-Moody Lie algebra is such an example, and
in early 1990's, a categorification of integrable modules over an affine Kac-Moody algebra of type $A^{(1)}_\ell$
was found to be useful for studying module categories of finite dimensional Hecke algebras. They were studied by the first author
and his collaborators. Then, generalization of this categorification to other Lie types by using cyclotomic quiver Hecke algebras,
which are also called cyclotomic Khovanov-Lauda-Rouquier algebras, has atrracted interests of various mathematicians in 2000's.

In any representation theory, classification of irreducible modules is the starting point.
But remaining at that stage is not very interesting and we must handle non-semisimple modules. To obtain a result,
one studies a special kind of non-semisimple modules, and various methods have been developed in each representation theory.

In our context, we are given new type of self-injective algebras, cyclotomic quiver Hecke algebras, and it is worth pursuing to develop new ways of obtaining results, namely,
proof techniques that combine those categorification of integrable modules over an affine Kac-Moody algebra with classical techniques from
representation theory of finite dimensional algebras. Indeed, we may expect that various properties of the usual Hecke algebras also
hold for the cyclotomic quiver Hecke algebras, and we may use the new proof techniques to prove them.
The property we consider in our series of papers is the representation type. For the classical Hecke algebras associated with the symmetric group,
there is a theorem by Erdmann and Nakano, so that we study its analogue in other affine Lie types, namely Erdmann-Nakano type theorems for
finite quiver Hecke algebras $\fqH(\beta)$. In the first paper \cite{AP12} of our series, we gave a general framework to determine the representation type,
which gave a new proof to the original Erdmann-Nakano theorem, and we showed that the representation type in affine
type $A^{(2)}_{2\ell}$ was also described by Lie theoretic terms.

In this paper, we study affine Lie type $D^{(2)}_{\ell+1}$. We follow the general strategy established in \cite{AP12}.
Recall that a key result in \cite{AP12} was explicit description of irreducible $\fqH(\delta-\alpha_i)$-modules.
They were used to show that $\fqH(\delta)$ is a Brauer tree algebra. The modules were
homogeneous modules. In type $D^{(2)}_{\ell+1}$, a key achievement is
our success in explicit description of irreducible $\fqH(2\delta-\alpha_i)$-modules. The modules are no more homogeneous and their construction required
more insight than $A^{(1)}_\ell$ and $A^{(2)}_{2\ell}$. Then we use the explicit description to show that
$\fqH(2\delta)$ is a symmetric special biserial algebra. 
We may determine the shape of the stable Auslander-Reiten quiver of finite quiver Hecke algebras of tame representation type.
Other parts of the arguments in \cite{AP12}
may be carried out for type $D^{(2)}_{\ell+1}$ with slight modifications.

The paper is organized as follows. \S1 and \S2 are for preliminaries. In \S3 we prove explicit dimension formulas for various
spherical subalgebras of $\fqH(n)$. The formulas are crucial for proving the main theorem, and they are repeatedly used in the remaining half
of the paper. In \S4 we construct various modules, and then we prove our main theorem in \S5.

We are grateful to Professor Skowro\'nski for his advice on special biserial algebras. Using his results
\cite{ES92} with Professor Erdmann in particular, we were able to determine the shape of the stable Auslander-Reiten quiver
of tame finite quiver Hecke algebras mentioned above.

\vskip 1em

\section{Preliminaries}

In this section, we briefly recall the combinatorial realization of the highest weight
$U_q(D_{\ell+1}^{(2)})$-crystal $B(\Lambda_0)$ using {\it Young walls} and the Fock space
of neutral
fermions in aspect of $U(D_{\ell+1}^{(2)})$-modules. We will use them to describe the categorification theory for
cyclotomic quiver Hecke algebras of type $D^{(2)}_{\ell+1}$.

\subsection{Cartan datum} \label{Sec: Cartan datum}

Let $I = \{0,1, \ldots, \ell \}$ be an index set, and let $\cmA$ be the {\it affine Cartan matrix} of type $D_{\ell+1}^{(2)}$ ($\ell \ge 2$)
$$ \cmA = (a_{ij})_{i,j\in I} = \left(
                                  \begin{array}{ccccccc}
                                    2  & -2 & 0  & \ldots & 0  & 0  & 0 \\
                                    -1 &  2 & -1 & \ldots & 0  & 0  & 0 \\
                                    0  & -1 & 2  & \ldots & 0  & 0  & 0 \\
                                    \vdots   &  \vdots  &  \vdots  & \ddots &  \vdots  &  \vdots  & \vdots \\
                                    0  & 0  & 0  & \ldots & 2  & -1 & 0 \\
                                    0  & 0  & 0  & \ldots & -1 & 2  & -1 \\
                                    0  & 0  & 0  & \ldots & 0  & -2 & 2 \\
                                  \end{array}
                                \right).
  $$
When $\ell=1$, the affine Cartan matrix of type $D_{2}^{(2)}$ is defined to be that of type $A_{1}^{(1)}$, namely
$$
\cmA = (a_{ij})_{i,j\in I} = \begin{pmatrix} 2  & -2 \\ -2 & 2 \end{pmatrix}.
$$
An {\it affine Cartan datum} $(\cmA, \wlP, \Pi, \Pi^{\vee})$ is given as
\begin{itemize}
\item[(1)] $\cmA$ is the affine Cartan matrix as above,
\item[(2)] $\wlP$ is a free abelian group of rank $\ell+2$, called the {\it weight lattice},
\item[(3)] $\Pi = \{ \alpha_i \mid i\in I \} \subset \wlP$, called the set of {\it simple roots},
\item[(4)] $\Pi^{\vee} = \{ h_i \mid i\in I\} \subset \wlP^{\vee} := \Hom( \wlP, \Z )$, called the set of {\it simple coroots},
\end{itemize}
which satisfy the following properties:
\begin{itemize}
\item[(a)] $\langle h_i, \alpha_j \rangle  = a_{ij}$ for all $i,j\in I$,
\item[(b)] $\Pi$ and $\Pi^{\vee}$ are linearly independent sets.
\end{itemize}

As in \cite[p.21]{HK02}, we fix a {\it scaling element} $d$ which obeys the condition $\langle d, \alpha_i\rangle=\delta_{i0}$,
and assume that $\Pi^{\vee}$ and $d$ form a $\Z$-basis of $P^{\vee}$. Then, the {\it fundamental weight} $\Lambda_0$ is defined by
$$\langle h_i, \Lambda_0\rangle=\delta_{i0},\quad \langle d,\Lambda_0\rangle=0.$$

We define another element $\mathsf{d} \in \mathsf{P}^{\vee}\otimes_\Z \Q$ by
$$\mathsf{d}=\sum_{i=1}^{\ell-1} ih_i+\frac{\ell}{2}h_\ell+2d.$$
Then, it satisfies
\begin{align} \label{Eq: elt d}
\langle \mathsf{d}, \Lambda_0 \rangle = 0, \quad \langle \mathsf{d}, \alpha_i \rangle = \left\{
                                                                                          \begin{array}{ll}
                                                                                            1 & \hbox{ if } i=0, \ell, \\
                                                                                            0 & \hbox{ otherwise.}
                                                                                          \end{array}
                                                                                        \right..
\end{align}

The free abelian group $\rlQ = \bigoplus_{i \in I} \Z \alpha_i$ is called the {\it root lattice},  and $\rlQ^+ = \sum_{i\in I} \Z_{\ge 0} \alpha_i$ is
the {\it positive cone} of the root lattice. For $\beta=\sum_{i \in I} k_i \alpha_i \in \rlQ^+$, set $|\beta|=\sum_{i \in I} k_i$ to be the {\it height} of $\beta$.
Let $\weyl$ be the {\it Weyl group} associated with $\cmA$, which is generated by $\{r_i\}_{i\in I}$ acting on $P$ by
$r_i\Lambda=\Lambda-\langle h_i, \Lambda\rangle\alpha_i$, for $\Lambda\in P$. The \emph{null root} of type $D^{(2)}_{\ell+1}$ is
$$ \delta = \alpha_0 + \alpha_1 + \cdots + \alpha_{\ell-1} + \alpha_\ell. $$
Note that $ \langle h_i, \delta\rangle = 0$ and $w \delta = \delta$, for $i\in I$ and $w\in \weyl$.
The standard symmetric bilinear pairing $(\ | \ )$ on $\wlP$ in type $D^{(2)}_{\ell+1}$ is
$$ ( \alpha_i | \Lambda ) = d_i \langle h_i , \Lambda \rangle \ \text{ for all } \Lambda \in \wlP,$$
where $(d_0,d_1, \ldots, d_\ell) = (1,2, \ldots, 2,1)$.

\subsection{Young walls} \label{Sec: Young walls}
In this subsection, we review a combinatorics of {\it Young walls} for type $D_{\ell+1}^{(2)}$.
A Young wall is a generalization of a colored Young diagram, which gives a combinatorial realization of crystals for basic representations of various quantum affine algebras \cite{HK02, Kang03,Kash91, Kash93}.

A Young wall of type $D_{\ell+1}^{(2)}$ is a wall consisting of the colored blocks below
\begin{align*}
\xy
(0,5.5)*{};(6,5.5)*{} **\dir{-};
(0,2.5)*{};(6,2.5)*{} **\dir{-};
(0,5.5)*{};(0,2.5)*{} **\dir{-};
(6,5.5)*{};(6,2.5)*{} **\dir{-};
(3.2,4.0)*{ _0};
(0,1.5)*{};(6,1.5)*{} **\dir{-};
(0,-1.5)*{};(6,-1.5)*{} **\dir{-};
(0,-1.5)*{};(0,1.5)*{} **\dir{-};
(6,-1.5)*{};(6,1.5)*{} **\dir{-};
(3.2,0)*{ _{\ell}};
(50,2)*{ : \text{ unit width and half-unit height, unit thickness,}};
(0,-4)*{};(6,-4)*{} **\dir{-};
(0,-10)*{};(6,-10)*{} **\dir{-};
(0,-4)*{};(0,-10)*{} **\dir{-};
(6,-4)*{};(6,-10)*{} **\dir{-};
(3.2,-7)*{ _i};
(60,-7.5)*{ (i=1,\ldots, \ell-1): \text{ unit width and unit height, unit thickness,}};
\endxy
\end{align*}
built by the following rules:
\begin{enumerate}
\item Blocks should be built in the pattern given below.
\item There should be no free space to the right of any block except the rightmost column.
\end{enumerate}
The pattern is given as follows:
$$
\xy
(0,-0.5)*{};(33,-0.5)*{} **\dir{.};
(0,-1)*{};(33,-1)*{} **\dir{.};
(0,-1.5)*{};(33,-1.5)*{} **\dir{.};
(0,-2)*{};(33,-2)*{} **\dir{.};
(0,-2.5)*{};(33,-2.5)*{} **\dir{.};
(0,-3)*{};(33,-3)*{} **\dir{-};
(0,0)*{};(33,0)*{} **\dir{-};
(0,3)*{};(33,3)*{} **\dir{-};
(0,9)*{};(33,9)*{} **\dir{-};
(0,19)*{};(33,19)*{} **\dir{-};
(0,25)*{};(33,25)*{} **\dir{-};
(0,28)*{};(33,28)*{} **\dir{-};
(0,31)*{};(33,31)*{} **\dir{-};
(0,37)*{};(33,37)*{} **\dir{-};
(0,47)*{};(33,47)*{} **\dir{-};
(0,53)*{};(33,53)*{} **\dir{-};
(0,56)*{};(33,56)*{} **\dir{-};
(0,59)*{};(33,59)*{} **\dir{-};
(0,65)*{};(33,65)*{} **\dir{-};
(33,-3)*{};(33,66)*{} **\dir{-};
(27,-3)*{};(27,66)*{} **\dir{-};
(21,-3)*{};(21,66)*{} **\dir{-};
(15,-3)*{};(15,66)*{} **\dir{-};
(9,-3)*{};(9,66)*{} **\dir{-};
(3,-3)*{};(3,66)*{} **\dir{-};
(30.2,1.5)*{_0}; (30.2,6)*{_1}; (30.2,15)*{\vdots}; (30.2,22)*{_{\ell-1}}; (30.2,26.5)*{_{\ell}}; (30.2,29.5)*{_{\ell}}; (30.2,34)*{_{\ell-1}}; (30.2,43)*{\vdots};
(30.2,50)*{_1}; (30.2,54.5)*{_0}; (30.2,57.5)*{_0}; (30.2,62)*{_1};
(24.2,1.5)*{_0}; (24.2,6)*{_1}; (24.2,15)*{\vdots}; (24.2,22)*{_{\ell-1}}; (24.2,26.5)*{_{\ell}}; (24.2,29.5)*{_{\ell}}; (24.2,34)*{_{\ell-1}}; (24.2,43)*{\vdots};
(24.2,50)*{_1}; (24.2,54.5)*{_0}; (24.2,57.5)*{_0}; (24.2,62)*{_1};
(18.2,1.5)*{_0}; (18.2,6)*{_1}; (18.2,15)*{\vdots}; (18.2,22)*{_{\ell-1}}; (18.2,26.5)*{_{\ell}}; (18.2,29.5)*{_{\ell}}; (18.2,34)*{_{\ell-1}}; (18.2,43)*{\vdots};
(18.2,50)*{_1}; (18.2,54.5)*{_0}; (18.2,57.5)*{_0}; (18.2,62)*{_1};
(12.2,1.5)*{_0}; (12.2,6)*{_1}; (12.2,15)*{\vdots}; (12.2,22)*{_{\ell-1}}; (12.2,26.5)*{_{\ell}}; (12.2,29.5)*{_{\ell}}; (12.2,34)*{_{\ell-1}}; (12.2,43)*{\vdots};
(12.2,50)*{_1}; (12.2,54.5)*{_0}; (12.2,57.5)*{_0}; (12.2,62)*{_1};
(6.2,1.5)*{_0}; (6.2,6)*{_1}; (6.2,15)*{\vdots}; (6.2,22)*{_{\ell-1}}; (6.2,26.5)*{_{\ell}}; (6.2,29.5)*{_{\ell}}; (6.2,34)*{_{\ell-1}}; (6.2,43)*{\vdots};
(6.2,50)*{_1}; (6.2,54.5)*{_0}; (6.2,57.5)*{_0}; (6.2,62)*{_1};
\endxy
$$
Note that the sequence $(0,1,2, \ldots, \ell-1, \ell, \ell, \ell-1, \ldots, 2,1,0)$ is repeated in each column.

A {\it full column} is a column whose height is a multiple of the unit length, and a {\it $\delta$-column} is a column consisting of two $0$-blocks, two $1$-blocks, $\ldots$, and two $\ell$-blocks.
The definition of the $\delta$-column might confuse the reader, but recall that $\delta$-column in type $D^{(2)}_{\ell+1}$ is a column of weight $2\delta$. See \cite[p.279]{HK02}.
A Young wall is called {\it proper} if none of full columns have the same height. A column in a proper Young wall is said to have a {\it removable} $\delta$ if we may remove a $\delta$-column from
the given Young wall and still have a proper Young wall.

Let $\mathcal{Y}(\Lambda_0)$ be the set of all proper Young walls whose columns have no removable $\delta$. Kashiwara operators $\tilde{e}_i$ and $\tilde{f}_i$ on $\mathcal{Y}(\Lambda_0)$ can be defined by considering combinatorics of Young walls,
which give a crystal structure of a quantum affine algebra $U_q(\g)$ of type $D_{\ell+1}^{(2)}$ \cite{HK02, Kang03}.
\begin{thm}[\protect{\cite[Thm.7.1]{Kang03}}]
The crystal $\mathcal{Y}(\Lambda_0)$ is isomorphic to the crystal $B(\Lambda_0)$ of the highest weight $U_q(\g)$-module $V_q(\Lambda_0)$.
\end{thm}

\begin{lemma}
\label{extremal weights}
$\Lambda_0-2\delta+\alpha_i$, for $0\le i\le \ell-1$, are extremal weights.
\end{lemma}
\begin{proof}
Explicit computation shows
\begin{align*}
r_1r_2\cdots r_\ell\cdots r_1r_0\Lambda_0&=\Lambda_0-2\delta+\alpha_0,\\
r_2\cdots r_\ell\cdots r_2r_0r_1r_0\Lambda_0&=\Lambda_0-2\delta+\alpha_1.
\end{align*}
Then $r_ir_{i+1}(\Lambda_0-2\delta+\alpha_i)=\Lambda_0-2\delta+\alpha_{i+1}$, for $1\le i\le \ell-2$, proves the result.
\end{proof}

Using the Young wall realization, one can observe the following facts:
\begin{enumerate}
\item[(i)] For each $i\in I$, $\mathcal{Y}(\Lambda_0)_{\Lambda_0-2\delta + \alpha_i}$ consists of the element $Y_i$ given below.
\begin{align} \label{Eq: Young walls Yi}
\xy
(0,-30)*{};(6,-30)*{} **\dir{-};
(0,-27)*{};(6,-27)*{} **\dir{-};
(0,-21)*{};(6,-21)*{} **\dir{-};
(0,-7)*{};(6,-7)*{} **\dir{-};
(0,-1)*{};(6,-1)*{} **\dir{-};
(0,2)*{};(6,2)*{} **\dir{-};
(0,5)*{};(6,5)*{} **\dir{-};
(0,11)*{};(6,11)*{} **\dir{-};
(0,27)*{};(6,27)*{} **\dir{-};
(0,33)*{};(6,33)*{} **\dir{-};
(0,-30)*{};(0,33)*{} **\dir{-};
(6,-30)*{};(6,33)*{} **\dir{-};
(-5,-5)*{Y_0 =\ \ };
(3.3,-28.5)*{_0 }; (3.3,-24)*{_1 }; (3.3,-13)*{\vdots }; (3.3,-4)*{_{\ell-1} };(3.3,0.5)*{_{\ell} }; (3.3,3.5)*{_{\ell} }; (3.3,8)*{_{\ell-1} };
(3.3,20)*{\vdots }; (3.3,30)*{_1 };
(30,-30)*{};(42,-30)*{} **\dir{-};
(30,-27)*{};(42,-27)*{} **\dir{-};
(30,-17)*{};(42,-17)*{} **\dir{-};
(30,-11)*{};(42,-11)*{} **\dir{-};
(36,-1)*{};(42,-1)*{} **\dir{-};
(36,2)*{};(42,2)*{} **\dir{-};
(36,5)*{};(42,5)*{} **\dir{-};
(36,15)*{};(42,15)*{} **\dir{-};
(36,21)*{};(42,21)*{} **\dir{-};
(30,-30)*{};(30,-11)*{} **\dir{-};
(36,-30)*{};(36,21)*{} **\dir{-};
(42,-30)*{};(42,21)*{} **\dir{-};
(25,-5)*{Y_i =\ \ }; (62,-5)*{(i=1,\ldots,\ell-1) };
(33.3,-28.5)*{_0 }; (39.3,-28.5)*{_0 };
(33.3,-21)*{\vdots }; (39.3,-21)*{ \vdots };
(33.3,-14)*{_{i-1} }; (39.3,-14)*{ _{i-1} };
(39.3,-5)*{ \vdots };
(39.3,0.5)*{ _{\ell} };
(39.3,3.5)*{ _{\ell} };
(39.3,11)*{ \vdots };
(39.3,18)*{ _{i+1} };
(100,-30)*{};(112,-30)*{} **\dir{-};
(100,-27)*{};(112,-27)*{} **\dir{-};
(100,-21)*{};(112,-21)*{} **\dir{-};
(100,-7)*{};(112,-7)*{} **\dir{-};
(100,-1)*{};(112,-1)*{} **\dir{-};
(106,2)*{};(112,2)*{} **\dir{-};
(100,-30)*{};(100,-1)*{} **\dir{-};
(106,-30)*{};(106,2)*{} **\dir{-};
(112,-30)*{};(112,2)*{} **\dir{-};
(95,-5)*{Y_\ell =\ \ };
(103.3,-28.5)*{ _{0} };(109.3,-28.5)*{ _{0} };
(103.3,-24)*{ _{1} };(109.3,-24)*{ _{1} };
(103.3,-13)*{ \vdots };(109.3,-13)*{ \vdots };
(103.3,-4)*{ _{\ell-1} };(109.3,-4)*{ _{\ell-1} };
(109.3,0.5)*{ _{\ell} };
\endxy
\end{align}
\item[(ii)] The Young walls $Y_0, \ldots Y_{\ell-1}$ are extremal.
\item[(iii)] We have $\varepsilon_\ell (Y_{\ell-1})= 2$ and if $i,j \in I$ are such that $(i,j) \ne (\ell-1,\ell)$ then
\begin{align} \label{Eq: properties of Yi}
\varepsilon_j(Y_i) = \left\{
                                                                                        \begin{array}{ll}
                                                                                          1 & \hbox{ if } j=i-1,i+1, \\
                                                                                          0 & \hbox{ otherwise}.
                                                                                        \end{array}
                                                                                      \right.
\end{align}
Moreover, in the case that $j$ is either $i-1$ or $i+1$, we have $\tilde{e}_i(Y_j) = \tilde{e}_j(Y_i)$.
\item[(iv)] We have
\begin{align} \label{Eq: num of irr R(2delta)-modules}
| \ \mathcal{Y}(\Lambda_0)_{\Lambda_0-2\delta}\ | = \ell+1.
\end{align}
\end{enumerate}

\subsection{Shifted Young diagrams}
Let $\lambda= (\lambda_1 > \lambda_2 > \ldots> \lambda_l > 0)$ be a shifted Young diagram of depth $l$. We denote the depth $l$ by $l(\lambda)$ and write $\lambda \vdash n$ if
$\lambda$ consists of $n$ boxes.
A {\it tableau} $T$ of shape $\lambda \vdash n$ is a filling of
boxes of $\lambda$ with numbers $1,\dots, n$, one for each box. For a simple transposition $s_k$,
let $s_kT$ be the tableau obtained from $T$ by exchanging the entries $k$ and $k+1$. A {\it standard tableau} is a tableau whose entries in rows and columns increase from left to right and top to bottom, respectively.
The {\it canonical tableau} $T^\lambda$ is the standard tableau whose $(i,j)$-entry is $(j-i) + 1 + \sum_{k=1}^{i-1} \lambda_k$.
We denote by $\ST(\lambda)$ the set of all standard tableaux of shape $\lambda$.
For example, the following are standard tableaux of shape $(5,2)$:
\vskip 0.3em
$$
\xy
(0,12)*{};(30,12)*{} **\dir{-};
(0,6)*{};(30,6)*{} **\dir{-};
(6,0)*{};(18,0)*{} **\dir{-};
(0,6)*{};(0,12)*{} **\dir{-};
(6,0)*{};(6,12)*{} **\dir{-};
(12,0)*{};(12,12)*{} **\dir{-};
(18,0)*{};(18,12)*{} **\dir{-};
(24,6)*{};(24,12)*{} **\dir{-};
(30,6)*{};(30,12)*{} **\dir{-};
(3,9)*{1}; (9,9)*{2}; (15,9)*{4}; (21,9)*{6}; (27,9)*{7};
(9,3)*{3}; (15,3)*{5};
(50,12)*{};(80,12)*{} **\dir{-};
(50,6)*{};(80,6)*{} **\dir{-};
(56,0)*{};(68,0)*{} **\dir{-};
(50,6)*{};(50,12)*{} **\dir{-};
(56,0)*{};(56,12)*{} **\dir{-};
(62,0)*{};(62,12)*{} **\dir{-};
(68,0)*{};(68,12)*{} **\dir{-};
(74,6)*{};(74,12)*{} **\dir{-};
(80,6)*{};(80,12)*{} **\dir{-};
(53,9)*{1}; (59,9)*{2}; (65,9)*{3}; (71,9)*{5}; (77,9)*{7};
(59,3)*{4}; (65,3)*{6};
\endxy
$$

Let $\lambda$ be a shifted Young diagram. We consider the residue pattern ($\ell\ge 1$)
\begin{equation} \label{Eq: residue pattern}
0\ 1\ 2\ \dots \ \ell-1\ \ell\ \ell \ \ell-1\ \dots\ 2\ 1\ 0,
\end{equation}
which repeats from left to right in each row of $\lambda$, and denote by $\res(i,j)$ the residue of the $(i,j)$-box in $\lambda$.
For example, if $\ell=3$ and $\lambda = (12,7,3,2)$, the residues are given as follows:
$$
\xy
(0,12)*{};(72,12)*{} **\dir{-};
(0,6)*{};(72,6)*{} **\dir{-};
(6,0)*{};(48,0)*{} **\dir{-};
(12,-6)*{};(30,-6)*{} **\dir{-};
(18,-12)*{};(30,-12)*{} **\dir{-};
(0,12)*{};(0,6)*{} **\dir{-};
(6,12)*{};(6,0)*{} **\dir{-};
(12,12)*{};(12,-6)*{} **\dir{-};
(18,12)*{};(18,-12)*{} **\dir{-};
(24,12)*{};(24,-12)*{} **\dir{-};
(30,12)*{};(30,-12)*{} **\dir{-};
(36,12)*{};(36,0)*{} **\dir{-};
(42,12)*{};(42,0)*{} **\dir{-};
(48,12)*{};(48,0)*{} **\dir{-};
(54,12)*{};(54,6)*{} **\dir{-};
(60,12)*{};(60,6)*{} **\dir{-};
(66,12)*{};(66,6)*{} **\dir{-};
(72,12)*{};(72,6)*{} **\dir{-};
(3,9)*{0}; (9,9)*{1}; (15,9)*{2}; (21,9)*{3}; (27,9)*{3}; (33,9)*{2}; (39,9)*{1};(45,9)*{0};(51,9)*{0}; (57,9)*{1}; (63,9)*{2}; (69,9)*{3};
           (9,3)*{0}; (15,3)*{1}; (21,3)*{2}; (27,3)*{3}; (33,3)*{3}; (39,3)*{2}; (45,3)*{1};
                      (15,-3)*{0}; (21,-3)*{1}; (27,-3)*{2};
                                   (21,-9)*{0}; (27,-9)*{1};
\endxy
$$
In this example, we have $\res(2,7) = 2$. Note that the residue pattern also appears in columns of Young walls as colors (see Section \ref{Sec: Young walls}).

\begin{defn} \label{Def: residue sequence}
For $T\in \ST(\lambda)$, we define the {\it residue sequence} of $T$ by
$$ \res(T) = (\res_1(T), \res_2(T), \ldots, \res_n(T) )\in I^n,$$
where $\res_k(T)$ is the residue of the box filled with $k$ in $T$, for $1\leq k\leq n$.
\end{defn}

\subsection{The Fock space of neutral fermions} \label{Sec: Fock space}

Let $\mathsf{C}$ be the Clifford algebra over $\C$ generated by $\phi_k\ (k\in \Z)$ with the following defining relations:
$$ \phi_p \phi_q + \phi_q \phi_p = \left\{
                                     \begin{array}{ll}
                                       2 & \hbox{ if } p=q=0, \\
                                       (-1)^p \delta_{p,-q} & \hbox{ otherwise.}
                                     \end{array}
                                   \right.
  $$
Consider the quotient space $\mathsf{F} = \mathsf{C} / \mathsf{I}$ by the left ideal $\mathsf{I}$ of $\mathsf{C}$ generated by $\{ \phi_k \mid k<0 \}$.
For a shifted Young diagram $\lambda$, we can write it as $\lambda=(\lambda_1 > \lambda_2 > \ldots > \lambda_{2r-1} > \lambda_{2r} \ge 0 )$ for a unique number $r$,
where $\lambda_{2r}>0$ if $l(\lambda)$ is even and $\lambda_{2r}=0$ if $l(\lambda)$ is odd. Set $|0\rangle = 1 + \mathsf{I}$ and
$$|\lambda \rangle = \phi_{\lambda_1}\phi_{\lambda_2} \cdots \phi_{\lambda_{2r}}|0\rangle. $$
Note that $\{ |\lambda \rangle \mid \lambda : \text{shifted Young diagrams} \}$ is a linearly independent set in $\mathsf{F}$.

We now define
$$\mathcal{F} = \Span_\C\{ |\lambda \rangle \mid \lambda : \text{shifted Young diagrams} \}\subseteq \mathsf{F},$$
and describe the actions of the Chevalley generators $f_i, e_i\ (i\in I)$ on $\mathcal{F}$ for the Kac-Moody algebra $\g$ associated with $\mathsf{A}$
by the reduction $(BKP)_{2\ell+2}$ in \cite{DJKM82, JM83} as follows.
\begin{align} \label{Eq: affine YD operations}
f_i = \sum_{j\ge0, \ j\equiv i, -i-1} \mathsf{f}_j,  \qquad
e_i = \left\{
        \begin{array}{ll}
          \mathsf{e}_0 + 2 \sum_{j>0, \  j\equiv 0, -1} \mathsf{e}_j & \hbox{ if } i=0, \\
          \sum_{j>0, \ j\equiv i, -i-1} \mathsf{e}_j & \hbox{ if } i = 1, \ldots, \ell-1, \\
          2\sum_{j>0, \ j\equiv \ell, \ell+1} \mathsf{e}_j & \hbox{ if } i = \ell,
        \end{array}
      \right.
\end{align}
where all congruences are taken modulo $h:= 2\ell+2$, and
\begin{align*}
\mathsf{e}_0 &=  \phi_{-1} \phi_{0}, \qquad \quad \ \  \mathsf{e}_j = (-1)^j \phi_{-j-1} \phi_{j} \qquad  \qquad  \qquad \qquad  (j\ge 1), \\
\mathsf{f}_0 &= \phi_{1} \phi_{0}, \qquad  \quad\ \ \ \ \mathsf{f}_j = (-1)^j \phi_{j+1} \phi_{-j} \qquad \qquad  \qquad \qquad  (j\ge 1).
\end{align*}
$\mathcal{F}$ has the highest weight $\Lambda_0$.
It is straightforward to check that if $\lambda$ has a row of length $j+1$ and does not have a row of length $j$ (resp. $\lambda$ has a row of length $j$ and does not have
a row of length $j+1$), then we have
\begin{align} \label{Eq: finite YD operations}
\mathsf{e}_j |\lambda\rangle = |\mathsf{e}_j\lambda\rangle, \quad (\text{resp.}\;\;\mathsf{f}_j |\lambda\rangle = |\mathsf{f}_j\lambda\rangle),
\end{align}
where $\mathsf{e}_j\lambda$ (resp.\ $\mathsf{f}_j\lambda$) is the shifted Young diagram obtained from $\lambda$ by removing the rightmost box of
the row of length $j+1$ (resp. adding a new box on the right of the row of length $j$).
Otherwise, we have $\mathsf{e}_j |\lambda\rangle=0$ (resp. $ \mathsf{f}_j |\lambda\rangle=0)$. For a shifted Young diagram $\lambda$,
$| \lambda \rangle$ is a weight vector whose weight is
\begin{align*}
\wt(\lambda):= \wt( | \lambda \rangle) = \Lambda_0 - \sum_{k\in \res(\lambda)} \alpha_k ,
\end{align*}
where $\res(\lambda)$ is a multiset whose elements are $\res(i,j)$ for all $(i,j)\in \lambda$.

\vskip 1em

\section{Quiver Hecke algebras}\label{Sec: quiver Hecke algs}
In this section, we introduce the quiver Hecke algebras, the main object of study. Then
we review the categorification of integrable highest weight modules and crystals using their
module categories. The results in this section will be used for proving our main theorems.
Throughout the paper, $\bR$ is an algebraically closed field, and algebras are unital associative $\bR$-algebras.

Let $(\cmA, \wlP, \Pi, \Pi^{\vee})$ be the affine Cartan datum from Section \ref{Sec: Cartan datum}, and take polynomials $\mathcal{Q}_{i,j}(u,v)\in\bR[u,v]$, for $i,j\in I$,
of the form
\begin{align*}
\mathcal{Q}_{i,j}(u,v) = \left\{
                 \begin{array}{ll}
                   \sum_{p(\alpha_i|\alpha_i)+q (\alpha_j|\alpha_j) + 2(\alpha_i|\alpha_j)=0} t_{i,j;p,q} u^pv^q & \hbox{if } i \ne j,\\
                   0 & \hbox{if } i=j,
                 \end{array}
               \right.
\end{align*}
where $t_{i,j;p,q} \in \bR$ are such that $t_{i,j;-a_{ij},0} \ne 0$ and $\mathcal{Q}_{i,j}(u,v) = \mathcal{Q}_{j,i}(v,u)$.
The symmetric group $\mathfrak{S}_n = \langle s_k \mid k=1, \ldots, n-1 \rangle$ acts on $I^n$ by place permutations.

\begin{defn} \
Let $\Lambda \in \wlP^+$. The {\it cyclotomic quiver Hecke algebra} $R^{\Lambda}(n)$ associated with polynomials $(\mathcal{Q}_{i,j}(u,v))_{i,j\in I}$ and the dominant integral weight $\Lambda$
is the $\Z$-graded  $\bR$-algebra defined by three sets of generators
$$\{e(\nu) \mid \nu = (\nu_1,\ldots, \nu_n) \in I^n\}, \;\{x_k \mid 1 \le k \le n\}, \;\{\psi_l \mid 1 \le l \le n-1\} $$
subject to the following relations:

\begin{align*}
& e(\nu) e(\nu') = \delta_{\nu,\nu'} e(\nu),\ \sum_{\nu \in I^{n}} e(\nu)=1,\
x_k e(\nu) =  e(\nu) x_k, \  x_k x_l = x_l x_k,\\
& \psi_l e(\nu) = e(s_l(\nu)) \psi_l,\  \psi_k \psi_l = \psi_l \psi_k \text{ if } |k - l| > 1, \\[5pt]
&  \psi_k^2 e(\nu) = \mathcal{Q}_{\nu_k, \nu_{k+1}}(x_k, x_{k+1}) e(\nu), \\[5pt]
&  (\psi_k x_l - x_{s_k(l)} \psi_k ) e(\nu) = \left\{
                                                           \begin{array}{ll}
                                                             -  e(\nu) & \hbox{if } l=k \text{ and } \nu_k = \nu_{k+1}, \\
                                                               e(\nu) & \hbox{if } l = k+1 \text{ and } \nu_k = \nu_{k+1},  \\
                                                             0 & \hbox{otherwise,}
                                                           \end{array}
                                                         \right. \\[5pt]
&( \psi_{k+1} \psi_{k} \psi_{k+1} - \psi_{k} \psi_{k+1} \psi_{k} )  e(\nu) \\[4pt]
&\qquad \qquad \qquad = \left\{
                                                                                   \begin{array}{ll}
\displaystyle \frac{\mathcal{Q}_{\nu_k,\nu_{k+1}}(x_k,x_{k+1}) -
\mathcal{Q}_{\nu_k,\nu_{k+1}}(x_{k+2},x_{k+1})}{x_{k}-x_{k+2}} e(\nu) & \hbox{if } \nu_k = \nu_{k+2}, \\
0 & \hbox{otherwise}, \end{array}
\right.\\[5pt]
& x_1^{\langle h_{\nu_1}, \Lambda \rangle} e(\nu)=0.
\end{align*}
\end{defn}

\bigskip
The $\Z$-grading on $R^\Lambda(n)$ is given as follows:
\begin{align*}
\deg(e(\nu))=0, \quad \deg(x_k e(\nu))= ( \alpha_{\nu_k} |\alpha_{\nu_k}), \quad  \deg(\psi_l e(\nu))= -(\alpha_{\nu_{l}} | \alpha_{\nu_{l+1}}).
\end{align*}

\vskip 1em

For $\beta\in \rlQ^+$ with $|\beta|=n$, we set $I^\beta = \{ \nu=(\nu_1, \ldots, \nu_n ) \in I^n \mid \sum_{k=1}^n\alpha_{\nu_k} = \beta \}$ and
$$ e(\beta) = \sum_{\nu \in I^\beta} e(\nu),$$
which is a central idempotent of $R^\Lambda(n)$ since $I^\beta$ is invariant under the action of $\mathfrak{S}_n$. Define
$$ R^\Lambda(\beta) = R^\Lambda(n) e(\beta). $$
We do not know whether $R^\Lambda(\beta)$ is indecomposable.
For an $R^{\Lambda}(\beta)$-module $M$, the {\it $q$-character} $\mathrm{ch}_q(M)$ is defined by
$$ \mathrm{ch}_q(M) := \sum_{\nu\in I^\beta} \dim_q( e(\nu) M ) \nu ,$$
where $ \dim_q( N ) := \sum_{i\in \Z} (\dim N_i) q^i $ for a graded vector space $N = \bigoplus_{i\in \Z}N_i$.

We will be interested in the special cyclotomic quiver Hecke algebra $\fqH(\beta)$ associated with the Cartan datum of type $D_{\ell+1}^{(2)}$,
which we call \emph{finite quiver Hecke algebras of type $D_{\ell+1}^{(2)}$}.

\begin{Rmk}
As is explained in the introduction, we focus on finite quiver Hecke algebras because they resemble the Hecke
algebras associated with the symmetric group, and we study
how far the properties which were found for the classical Hecke
algebras remain true for finite quiver Hecke algebras. For example, we
have already shown that the representation type follows Erdmann and
Nakano's pattern in types $A^{(1)}_\ell, A^{(2)}_{2\ell}$ and
$D^{(2)}_{\ell+1}$, but in the fourth paper in our series, we will show that
we may still use our method but the representation type follows different pattern in type $C^{(1)}_\ell$.
The Erdmann and Nakano's pattern also fails for the level two and higher level cyclotomic
Hecke algebras.
\end{Rmk}

Before studying the finite quiver Hecke algebras of type $D_{\ell+1}^{(2)}$, we recall results which are valid for general $R^\Lambda(\beta)$.

\begin{prop} [\protect{cf. \cite[Lem.2.2]{AP12}}]  \label{Prop: finite-dim} The algebra $R^\Lambda(\beta)$ is finite-dimensional and $x_1, \ldots, x_n$ are nilpotent.
\end{prop}

\begin{prop} [\protect{cf. \cite[Prop.4.2]{AP12}}] \label{Prop: self-injective}
The algebra $R^\Lambda(\beta)$ is a self-injective algebra.
\end{prop}

\begin{prop} [\protect{\cite[Cor.4.8]{AP12}}] \label{Prop: repn type}
For $w \in \weyl$ and $k\in \Z_{\ge0}$, $R^\Lambda(k\delta)$ and $R^\Lambda(\Lambda - w\Lambda + k\delta)$
have the same number of simple modules and the same representation type.
\end{prop}

Let $\A = \Z[q, q^{-1}]$. We denote by $U_\A(\g)$ the $\A$-form of the quantum group $U_q(\g)$ associated with the Cartan datum $(\cmA, \wlP, \Pi, \Pi^{\vee})$ and
by $V_\A(\Lambda)$ the $\A$-form of highest weight $U_q(\g)$-module $V_q(\Lambda)$, respectively.

Let us denote the direct sum of split Grothendieck groups of additive categories $R^\Lambda(\beta)\text{-}\proj^\Z$ of finitely generated projective
graded left $R^\Lambda(\beta)$-modules by
\begin{align*}
K^\Z_0(R^{\Lambda})=\bigoplus_{\beta \in \rlQ^+} K_0(R^{\Lambda}(\beta)\text{-}\proj^\Z).
\end{align*}
Then, $K^\Z_0(R^{\Lambda})$ has an $\A$-module structure induced by the $\Z$-grading on $R^\Lambda(n)$.

Let $e(\nu,\nu')$ be the idempotent corresponding to the concatenation of $\nu$ and $\nu'$.
When $\nu'=i$, we simply write $e(\nu,i) = e(\nu,(i))$ and $e(\beta,i) = \sum_{\nu\in I^\beta} e(\nu, i)$. For $\beta\in \rlQ^+$ and $i\in I$, we define functors between categories of finitely generated graded modules as follows:
\begin{align*}
E_i &: R^\Lambda(\beta + \alpha_i)\text{-mod}^\Z \longrightarrow R^\Lambda(\beta)\text{-mod}^\Z,  \quad N \mapsto e(\beta,i)N, \\
F_i &: R^\Lambda(\beta )\text{-mod}^\Z \longrightarrow R^\Lambda(\beta+ \alpha_i)\text{-mod}^\Z, \quad M \mapsto R^{\Lambda}(\beta+\alpha_i) e(\beta,i)\otimes_{R^{\Lambda}(\beta)}M.
\end{align*}

\begin{thm} [\protect{\cite[Thm.5.2]{KK11}}] \label{Thm: Ei Fi}
Let $l_i = \langle h_i, \Lambda - \beta  \rangle$, for $i\in I$. Then one of the following  isomorphisms of endofuctors holds.
\begin{enumerate}
\item If $l_i \ge 0$, then
$$ q_i^{-2}F_i E_i \oplus \bigoplus_{k=0}^{l_i- 1} q_i^{2k} \mathrm{id} \buildrel \sim \over \longrightarrow E_iF_i .$$
\item If $l_i \le 0$, then
$$ q_i^{-2}F_i E_i  \buildrel \sim \over \longrightarrow  E_iF_i \oplus \bigoplus_{k=0}^{-l_i- 1} q_i^{-2k-2}  \mathrm{id} .$$
\end{enumerate}
\end{thm}
Moreover, $E_i$ and $F_i$ give a $U_\A(\g)$-module structure to $K^\Z_0(R^{\Lambda})$.

\begin{thm}[\protect{\cite[Thm.6.2]{KK11}}] \label{Thm: categorification thm}
There exists a $U_\A(\g)$-module isomorphism between $K^\Z_0(R^{\Lambda})$ and $V_\A(\Lambda)$.
\end{thm}

Specializing $q \rightarrow 1$ on both sides of the isomorphism in Theorem \ref{Thm: categorification thm}, we have a $\g$-module isomorphism
\begin{align} \label{Eq: specialize q->1}
V_\Z(\Lambda) \simeq K_0(R^\Lambda) := \bigoplus_{\beta \in \rlQ^+} K_0(R^\Lambda(\beta)\text{-proj}),
\end{align}
where $V_\Z(\Lambda)$ is the Kostant $\Z$-form of the highest weight module $V(\Lambda)$, and $R^\Lambda(\beta)\text{-proj}$ is the category of
finitely generated projective $R^\Lambda(\beta)$-modules.

By \cite[Thm.3.5]{Kash11} and \cite[Lem.3.3]{AP12}, the Shapovalov form $(\ ,\ )$ of $V_\Z(\Lambda)$
can be understood as dimension of the space of homomorphisms between projective modules. Namely, if $u,v \in V_\Z(\Lambda)_{\Lambda-\beta} $
correspond to actual projective $R^\Lambda(\beta)$-modules $U, V \in K_0(R^\Lambda(\beta)\text{-proj})$ under the isomorphism $\eqref{Eq: specialize q->1}$, then we have
\begin{align} \label{Eq: Shapovalov form}
(u,v) = \dim \Hom_{R^{\Lambda}(\beta)}(U,V).
\end{align}

Let $\sB(\Lambda)$ be the set of all isomorphism classes of irreducible graded $R^{\Lambda}(\beta)$-modules, for all $\beta\in \mathsf{Q}^+$.
For $i\in I$ and $ L \in \sB(\Lambda)$, the Kashiwara operators are defined by
$$ \tilde{e}_i(L) = \Soc(E_i L)\quad \text{ and }\quad \tilde{f}_i(L) = \Top ( F_i L ) , $$
which give a $U_q(\g)$-crystal structure on $\sB(\Lambda)$.
\begin{thm}[\protect{\cite[Thm.7.5]{LV11}}] \label{Thm: crystals}
The crystal $\sB(\Lambda)$ is isomorphic to the crystal $B(\Lambda)$ of the highest weight $U_q(\g)$-module $V_q(\Lambda)$.
\end{thm}

\vskip 1em

\section{Dimension formula for $\fqH(\beta)$}
In this section, we give a dimension formula for finite quiver Hecke algebras $\fqH(\beta)$
in terms of combinatorics of shifted Young diagrams.
The Fock space $\F$ in
Section \ref{Sec: Fock space} and the categorication $\eqref{Eq: specialize q->1}$
in Section \ref{Sec: quiver Hecke algs} are key ingredients for deducing the dimension formula.

For a shifted Young diagram $\lambda$,
\begin{itemize}
\item if we may remove a box of residue $i$ from $\lambda$ and get a new shifted Young diagram, then we write $\lambda \nearrow \sBox{i}$ for the resulting diagram,
\item if we may add a box of residue $i$ to $\lambda$ and get a new shifted Young diagram, then we write $\lambda \swarrow \sBox{i}$ for the resulting diagram.
\end{itemize}
Then, it follows from $\eqref{Eq: affine YD operations}$ and $\eqref{Eq: finite YD operations}$ that
\begin{align} \label{Eq: Kashiwara operators}
e_i |\lambda\rangle = \sum_{\mu = \lambda \nearrow \ssBox{i} } \mathrm{m}(\lambda, \mu) |\mu\rangle, \qquad f_i |\lambda\rangle = \sum_{\mu = \lambda \swarrow \ssBox{i} } |\mu\rangle,
\end{align}
where $$ \mathrm{m}(\lambda, \mu) = \left\{
                           \begin{array}{ll}
                             2 & \hbox{ if either ($\wt(\mu)= \wt(\lambda)+\alpha_0$ and $l(\lambda)=l(\mu)$) or ($\wt(\mu)= \wt(\lambda)+\alpha_\ell$),}  \\
                             1 & \hbox{ otherwise.}
                           \end{array}
                         \right.
$$
The coefficients appearing in $\eqref{Eq: Kashiwara operators}$ are different from those in \cite[(3.1.1)]{AP12}.
The reduction processes using the Fock spaces for type $A_{2\ell}^{(2)}$ and type $D_{\ell+1}^{(2)}$ give that difference,
which yields different dimension formulas for $\fqH(\beta)$.

For $\lambda \vdash n$ and $\nu \in I^n$, we set
$$ K(\lambda, \nu) = | \{ T\in \ST(\lambda) \mid \nu = \res(T)  \} | $$
so that
\begin{align} \label{Eq: sum of K}
|\ST(\lambda)| = \sum_{\nu \in I^{\Lambda_0 - \wt(\lambda)} }  K(\lambda, \nu).
\end{align}
Recall the element $\mathsf{d}$ defined in $\eqref{Eq: elt d}$. Note that, if $\wt(\lambda) = \Lambda_0 - \sum_{i\in I}k_i\alpha_i$,
then $k_0+k_\ell = -\langle \mathsf{d}, \wt(\lambda) \rangle $.
Then, the following Lemma \ref{Lem: dimension formula}, Theorem \ref{Thm: dimension formula} and Corollary \ref{Cor: dimension formula}
can be proved in the same manner as \cite[Lem.3.2, Thm.3.4, Cor.3.5]{AP12}.

\begin{lemma} \label{Lem: dimension formula}
For $\lambda \vdash n$ and $\nu = (\nu_1, \nu_2, \ldots, \nu_n)\in I^n$, we obtain
\begin{align*}
e_{\nu_1}e_{\nu_2} \cdots e_{\nu_n} |\lambda\rangle &= 2^{-\langle \mathsf{d}, \wt(\lambda)\rangle - l(\lambda)}  K(\lambda, \nu) |0\rangle, \\
f_{\nu_n}f_{\nu_{n-1}} \cdots f_{\nu_1} |0\rangle &= \sum_{\mu \vdash n} K(\mu, \nu)  |\mu\rangle.
\end{align*}
\end{lemma}

\begin{thm} \label{Thm: dimension formula}
Let $\lambda \vdash n$ be a shifted Young diagram. For $\beta \in \rlQ^+$ with $|\beta|=n$ and $\nu, \nu' \in I^\beta$, we have
\begin{align*}
\dim e(\nu')\fqH(n)e(\nu) &= \sum_{\lambda \vdash n} 2^{-\langle \mathsf{d}, \wt(\lambda) \rangle - \l(\lambda)} K(\lambda, \nu') K(\lambda, \nu), \\
\dim \fqH(\beta) &= \sum_{\lambda \vdash n,\ \wt(\lambda)= \Lambda_0 - \beta} 2^{-\langle \mathsf{d}, \wt(\lambda) \rangle - \l(\lambda)} |\ST(\lambda)|^2, \\
\dim \fqH(n) &= \sum_{\lambda \vdash n} 2^{-\langle \mathsf{d}, \wt(\lambda) \rangle - \l(\lambda)} |\ST(\lambda)|^2.
\end{align*}
\end{thm}

\begin{cor} \label{Cor: dimension formula}
\begin{enumerate}
\item Let $\nu \in I^n$. Then, $e(\nu) \ne 0 $ in $\fqH(n)$ if and only if $\nu$ may be obtained from a standard tableau $T$ as $\nu = \res(T)$.
\item We have the following hook length formula.
$$ \dim \fqH(\beta) = \sum_{\lambda \vdash n,\ \wt(\lambda)=\Lambda_0 - \beta} 2^{-\langle \mathsf{d}, \wt(\lambda) \rangle - l(\lambda)} \left( \frac{n!}{\prod_{(i,j)\in \lambda}h_{i,j}} \right)^2, $$
where $h_{i,j}$ is the hook length of $(i,j) \in \lambda$.
\item For a natural number $n$, we have
$$ n! = \sum_{\beta \in \rlQ^+,\ |\beta|=n} 2^{n -\langle \mathsf{d}, \beta \rangle }\dim \fqH(\beta). $$
\end{enumerate}
\end{cor}

\begin{Rmk}
When $\ell=1$, Theorem \ref{Thm: dimension formula} recovers the following classical combinatorial formula.
$$ n! = \sum_{\lambda \vdash n} 2^{n-l(\lambda)}|\ST(\lambda)|^2.$$
The proof goes as follows. Suppose that $\ell=1$. Then,
$ -\langle \mathsf{d}, \wt(\lambda) \rangle - \l(\lambda) = n-l(\lambda)$
and Theorem \ref{Thm: dimension formula} implies that
$$ \dim \fqH(n) = \sum_{\lambda \vdash n} 2^{n-l(\lambda)}|\ST(\lambda)|^2. $$
On the other hand, since $D^{(2)}_{\ell+1}$ for $\ell=1$ is $A_1^{(1)}$,
there exists an isomorphism of algebras between $\fqH(n)$ and the finite Hecke algebra
$\mathcal{H}_{-1}(S_n)$ associated with the symmetric group $S_n$ and the parameter $q=-1$ \cite{BK09,R08}. Thus, $ \dim \fqH(n) =n!$ and we are done.

Note that the proof of Corollary \ref{Cor: dimension formula}(3) uses the formula, so that we must avoid
its use in the above proof. For this reason, we use a different argument to relate $\dim \fqH(n)$ and $n!$.
\end{Rmk}

\vskip 1em

\section{Representations of $\fqH(\beta)$} \label{Sec: repns}

In this section, we give various representations for the algebras $\fqH(\delta)$, $\fqH(2\delta-\alpha_i)$ and $\fqH(2\delta)$.
Those representations will play crucial roles in proving our main theorem in \S\ref{Sec: repn type}. We assume that $\ell \ge 2 $.
We also suppose that $\mathcal{Q}_{0, 1}(u,v) = u^2 + v$, $\mathcal{Q}_{\ell-1, \ell}(u,v) = u + v^2$ and
$\mathcal{Q}_{i, j}(u,v) = 1$ if $a_{i,j}=0$, for simplicity.

\subsection{Representations of $\fqH(\delta)$} \label{finite case}
We show that the algebra homomorphism
$$ \bR[x]\to \fqH(\delta):\;\; x\mapsto x_{\ell+1}$$
induces an isomorphism of algebras $\bR[x]/(x^2) \simeq \fqH(\delta)$.

Observe that the residue pattern $\eqref{Eq: residue pattern}$ implies that
$\lambda =(\ell+1)$ is the unique shifted Young diagram of weight $\Lambda_0-\delta$, and
$\ST(\lambda)$ consists of the standard tableau
\begin{align} \label{Eq: tableau of fqH(delta)}
\xy
(0,6)*{};(34,6)*{} **\dir{-};
(0,0)*{};(34,0)*{} **\dir{-};
(0,0)*{};(0,6)*{} **\dir{-};
(6,0)*{};(6,6)*{} **\dir{-};
(12,0)*{};(12,6)*{} **\dir{-};
(22,0)*{};(22,6)*{} **\dir{-};
(28,0)*{};(28,6)*{} **\dir{-};
(34,0)*{};(34,6)*{} **\dir{-};
(3,3)*{_0}; (9,3)*{_1}; (17,3)*{\cdots}; (25.3,3)*{_{\ell-1}}; (31,3)*{_{\ell}}; (35.5,0)*{.};
\endxy
\end{align}
Hence, Theorem \ref{Thm: dimension formula} implies that $\dim \fqH(\delta) = 2$. Now, it is straightforward to check that
\begin{align*}
 e(\nu) \mapsto & \left\{
                \begin{array}{ll}
                  (\begin{smallmatrix} 1 & 0 \\ 0 & 1 \end{smallmatrix}) & \hbox{ if } \nu = (0,1, \ldots,  \ell), \\
                  (\begin{smallmatrix} 0 & 0 \\ 0 & 0 \end{smallmatrix}) & \hbox{ otherwise,}
                \end{array}
              \right. \\
\psi_j \mapsto & (\begin{smallmatrix} 0 & 0 \\ 0 & 0 \end{smallmatrix}), \\
x_i \mapsto &  \left\{
                    \begin{array}{ll}
                      (\begin{smallmatrix} 0 & 1 \\ 0 & 0 \end{smallmatrix}) & \hbox{ if } i =\ell+1, \\
                      (\begin{smallmatrix} 0 & 0 \\ 0 & 0 \end{smallmatrix}) & \hbox{ otherwise.}
                    \end{array}
                  \right.
\end{align*}
is a well-defined representation of $\fqH(\delta)$. Hence, we have the algebra isomorphism
$$
\fqH(\delta)\simeq \bR(\begin{smallmatrix} 1 & 0 \\ 0 & 1 \end{smallmatrix})\oplus \bR(\begin{smallmatrix} 0 & 1 \\ 0 & 0 \end{smallmatrix})
\simeq \bR[x]/(x^2).
$$

\subsection{Representations of $\fqH(2\delta-\alpha_i)$} \label{Section: repn of R(2delta-alpha)}
In this subsection, we study $\fqH(2\delta-\alpha_i)$. A key point here is that the irreducible $\fqH(2\delta-\alpha_i)$-module is unique
and we may give its explicit description, for all $i\in I$. We will use
the result to determine the structure of $\fqH(2\delta)$.

Let $h= 2\ell+2$ as before. We consider two rows partition $\lambda(i)= (h-i-1, i ) \vdash h-1$ of weight $\Lambda_0 - 2\delta + \alpha_i$,
for $0\le i\le\ell$. Then the residues of the nodes of $\lambda(i)$ are given as follows:
\begin{equation} \label{Eq: lambda i}
\begin{aligned}
\lambda(i) = \left\{
               \begin{array}{ll}
\xy
(0,6)*{};(66,6)*{} **\dir{-};
(0,0)*{};(66,0)*{} **\dir{-};
(6,-6)*{};(28,-6)*{} **\dir{-};
(0,0)*{};(0,6)*{} **\dir{-};
(6,-6)*{};(6,6)*{} **\dir{-};
(12,-6)*{};(12,6)*{} **\dir{-};
(22,-6)*{};(22,6)*{} **\dir{-};
(28,-6)*{};(28,6)*{} **\dir{-};
(38,0)*{};(38,6)*{} **\dir{-};
(44,0)*{};(44,6)*{} **\dir{-};
(50,0)*{};(50,6)*{} **\dir{-};
(60,0)*{};(60,6)*{} **\dir{-};
(66,0)*{};(66,6)*{} **\dir{-};
(3,3)*{_0}; (9,3)*{_1}; (17,3)*{\cdots}; (25.3,3)*{_{i}}; (33,3)*{\cdots}; (41,3)*{_\ell}; (47,3)*{_\ell}; (55,3)*{\cdots}; (63.3,3)*{_{i+1}};
(9,-3)*{_0}; (17,-3)*{\cdots}; (25.3,-3)*{_{i-1}};
\endxy
& \hbox{ if } i = 0,1,\ldots, \ell-1, \\ \\
\xy
(0,6)*{};(28,6)*{} **\dir{-};
(0,0)*{};(28,0)*{} **\dir{-};
(6,-6)*{};(28,-6)*{} **\dir{-};
(0,0)*{};(0,6)*{} **\dir{-};
(6,-6)*{};(6,6)*{} **\dir{-};
(12,-6)*{};(12,6)*{} **\dir{-};
(22,-6)*{};(22,6)*{} **\dir{-};
(28,-6)*{};(28,6)*{} **\dir{-};
(3,3)*{_0}; (9,3)*{_1}; (17,3)*{\cdots}; (25.3,3)*{_{\ell}};
(9,-3)*{_0}; (17,-3)*{\cdots}; (25.3,-3)*{_{\ell-1}};
\endxy
& \hbox{ if } i=\ell.
               \end{array}
             \right.
\end{aligned}
\end{equation}
It is easy to check that $T\in \ST(\lambda(i))$ is uniquely determined by $\res(T)$. Let $\mathcal{N}=\bR^2$ be the two dimensional vector space
and we define
$$
\xi_1= \begin{pmatrix} 0 & 1 \\ 0 & 0 \end{pmatrix},\quad
\xi_2= \begin{pmatrix} 0 & -1 \\ 0 & 0 \end{pmatrix}, \quad
\partial= \begin{pmatrix} 0 & 0 \\ -1 & 0 \end{pmatrix}.
$$

\begin{Rmk}
We have $\xi_1\partial-\partial\xi_2=-1=\partial\xi_1-\xi_2\partial$, $\xi_1^2=0$, $\xi_1\xi_2=\xi_2\xi_1$ and $\partial^2=0$. These are defining relations for
the cyclotomic nilHecke algebra $NH_2^2$. Indeed,
$NH_2^2$ is isomorphic to $\Mat(2,\bR)$ \cite[Prop.5.3]{Lauda11}, and $\mathcal{N}$ is the unique irreducible $NH_2^2$-module.
\end{Rmk}

We shall define $\fqH(2\delta-\alpha_i)$-modules $\mathcal{L}_i$, for $0\le i\le\ell$. Let
$$ \mathcal{L}_i = \bigoplus_{T\in\ST(\lambda(i))} \mathcal{N} \otimes T$$
be the direct sum of copies of $\mathcal{N}$ such that each of the copies are labeled by $\ST(\lambda(i))$.

\begin{lemma}
If $0\le i\le\ell-1$, we may define an $\fqH(2\delta-\alpha_i)$-module structure on $\mathcal{L}_i$ by
\begin{equation} \label{Eq: Li}
\begin{aligned}
 e(\nu) (v\otimes T) &= \left\{
                \begin{array}{ll}
                  v\otimes T & \hbox{ if } \nu = \res(T), \\
                  0 & \hbox{ otherwise,}
                \end{array}
              \right. \\
x_k (v\otimes T) &= \left\{
                \begin{array}{ll}
                  \xi_1v\otimes T & \hbox{ if $k$ is located at $(1,\ell+1)\in\lambda(i)$}, \\
                  \xi_2v\otimes T & \hbox{ if $k$ is located at $(1,\ell+2)\in\lambda(i)$}, \\
                  0 & \hbox{ otherwise,}
                \end{array}
              \right. \\
\psi_k (v\otimes T) &= \left\{
                         \begin{array}{ll}
                           \partial v \otimes T & \hbox{ if $k$ and $k+1$ are located at $(1,\ell+1)$ and $(1,\ell+2)$}, \\
                           v \otimes s_k T  & \hbox{ if $s_k T$ is standard,}   \\
                           0 & \hbox{ otherwise.}
                         \end{array}
                       \right.
\end{aligned}
\end{equation}
If $i=\ell$, we may define an $\fqH(2\delta-\alpha_\ell)$-module structure on $\mathcal{L}_\ell$ by
\begin{equation}
\begin{aligned}
 e(\nu) (v\otimes T) &= \left\{
                \begin{array}{ll}
                  v\otimes T & \hbox{ if } \nu = \res(T), \\
                  0 & \hbox{ otherwise,}
                \end{array}
              \right. \\
x_k (v\otimes T) &= \left\{
                \begin{array}{ll}
                  \xi_1v\otimes T & \hbox{ if $k$ is located at $(1,\ell+1)\in\lambda(\ell)$}, \\
                  0 & \hbox{ otherwise,}
                \end{array}
              \right. \\
\psi_k (v\otimes T) &= \left\{
                         \begin{array}{ll}
                           v \otimes s_k T  & \hbox{ if $s_k T$ is standard,}   \\
                           0 & \hbox{ otherwise.}
                         \end{array}
                       \right.
\end{aligned}
\label{reducible}
\end{equation}
\end{lemma}
\begin{proof}
To check the defining relations on $v\otimes T$, we may assume that
$$\nu=(\nu_1, \nu_2, \ldots, \nu_{h-1})=\res(T).$$
Then, using the residue pattern $\eqref{Eq: residue pattern}$ and combinatorics of standard tableaux, we have
\begin{enumerate}
\item[(i)] If $\nu_k = \nu_{k+1}$, then $0\le i\le\ell-1$ and $\nu_k = \nu_{k+1}=\ell$.
\item[(ii)] $s_k T$ is standard if and only if $a_{\nu_k, \nu_{k+1}} = 0$.
\item[(iii)] If $\nu_k = \nu_{k+2}$, then $(\nu_k, \nu_{k+1},  \nu_{k+2})$ is either $(0,1,0)$ or $(\ell, j, \ell)$ for some $j\in I$ with $a_{\ell,j}=0$.
\end{enumerate}

In the rest of the proof, we only check the relations for $(\psi_k x_l - x_{s_k(l)} \psi_k)e(\nu)$, $\psi_k^2e(\nu)$ and
$(\psi_{k+1}\psi_{k}\psi_{k+1} - \psi_{k}\psi_{k+1}\psi_{k})e(\nu)$, because it is straightforward to check other relations.

Let us start with $(\psi_k x_l - x_{s_k(l)} \psi_k )e(\nu)(v\otimes T)$. Suppose that $\nu_k=\nu_{k+1}$. Then (i) implies that
$k$ and $k+1$ are located at the nodes $(1,\ell+1)$ and $(1,\ell+2)$ of $\lambda(i)$. If $l\ne k, k+1$ then
$$x_l(v\otimes T)=0, \quad x_{s_k(l)} \psi_k(v\otimes T)=x_l(\partial v\otimes T)=0.$$
Thus, $(\psi_k x_l - x_{s_k(l)} \psi_k )e(\nu)(v\otimes T)=0$ as desired. If $l=k$ then
$$(\psi_k x_k-x_{k+1} \psi_k)e(\nu)(v\otimes T)=\partial\xi_1v\otimes T-\xi_2\partial v\otimes T=-v\otimes T=-e(\nu)(v\otimes T).$$
If $l=k+1$, the similar computation shows $(\psi_k x_{k+1} -x_k \psi_k)e(\nu)(v\otimes T)=e(\nu)(v\otimes T)$.
Next we suppose that $\nu_k\ne\nu_{k+1}$ and prove $(\psi_k x_l -x_{s_k(l)} \psi_k)e(\nu)(v\otimes T)=0$.
If $s_k T$ is not standard then $\psi_k(\mathcal{N}\otimes T)=0$ implies the result. If $s_k T$ is standard, we define
$$\eta_l(T)= \# \{ 1 \le t \le l \mid \res_{t}(T) = \res_{l}(T) \}.$$
Then $\eta_l(T)$ is either $1$ or $2$ and we may write
$$ x_l (v\otimes T) = \delta_{\ell, \res_l(T)}  \xi_{\eta_l(T)} v \otimes  T. $$
Since $\res_{l}(T) = \res_{s_k(l)}(s_k T)$ holds and $\nu_k\ne\nu_{k+1}$ implies $\eta_{l}(T) = \eta_{s_k(l)}(s_k T)$,
we have
$$
(\psi_k x_l - x_{s_k(l)} \psi_k)e(\nu)(v\otimes T) =
\left(\delta_{\ell, \res_l(T)} \xi_{\eta_l(T)} - \delta_{\ell, \res_{s_k(l)}(s_kT)} \xi_{\eta_{s_k(l)}(s_k T)}\right)( v \otimes s_k T ) = 0.
$$

Next we prove the relation for $\psi_k^2e(\nu)$. If $s_k T$ is standard, then $a_{\nu_k, \nu_{k+1}}=0$ by (ii),
so that $\mathcal{Q}_{\nu_k,\nu_{k+1}}(x_k,x_{k+1})=1$ and
$\psi_k^2e(\nu)(v\otimes T)=\mathcal{Q}_{\nu_k,\nu_{k+1}}(x_k,x_{k+1})e(\nu)(v\otimes T)$ holds. If $s_k T$ is not standard, then
$\psi_k^2e(\nu)(v\otimes T)=0$ and we have to verify $\mathcal{Q}_{\nu_k,\nu_{k+1}}(x_k,x_{k+1})e(\nu)(v\otimes T)=0$.
If $\nu_k\ne\ell$ and $\nu_{k+1}\ne\ell$, it is clear. Since $s_k T$ is not standard, the remaining cases are
$(\nu_{k}, \nu_{k+1})=(\ell,\ell)$ or $(\nu_{k}, \nu_{k+1})=(\ell,\ell-1), (\ell-1,\ell)$. If $(\nu_{k}, \nu_{k+1})=(\ell,\ell)$
then $\mathcal{Q}_{\nu_k,\nu_{k+1}}(x_k,x_{k+1})=0$ and there is nothing to prove.
Suppose that $(\nu_{k}, \nu_{k+1})=(\ell-1,\ell)$. Then
$$ \mathcal{Q}_{\ell-1, \ell}(x_k, x_{k+1})e(\nu)(v\otimes T) = (x_k + x_{k+1}^2) (v\otimes T) = \xi_{\eta_{k+1}(T)}^2 v \otimes T = 0. $$
The proof for the case $(\nu_{k}, \nu_{k+1})=(\ell,\ell-1)$ is similar.

Finally, we consider $(\psi_{k+1}\psi_{k}\psi_{k+1} - \psi_{k}\psi_{k+1}\psi_{k})e(\nu)(v\otimes T)$. If
$\psi_{k+1}\psi_k\psi_{k+1}(v\otimes T)\ne 0$ then one of the following holds and the other does not hold.
\begin{itemize}
\item[(a)]
$s_{k+1}T, s_ks_{k+1}T, s_{k+1}s_ks_{k+1}T$ are all standard.
\item[(b)]
$\nu_k=\nu_{k+1}=\ell$ or $\nu_k=\nu_{k+2}=\ell$ or $\nu_{k+1}=\nu_{k+2}=\ell$.
\end{itemize}
Similarly, if $\psi_k\psi_{k+1}\psi_k(v\otimes T)\ne 0$ then one of the following holds and the other does not.
\begin{itemize}
\item[(a)]
$s_kT,s_{k+1}s_kT, s_ks_{k+1}s_kT$ are all standard.
\item[(b)]
$\nu_k=\nu_{k+1}=\ell$ or $\nu_k=\nu_{k+2}=\ell$ or $\nu_{k+1}=\nu_{k+2}=\ell$.
\end{itemize}
Note that $s_kT$, $s_{k+1} s_kT$, $s_k s_{k+1} s_kT$ are all standard if and only if $s_{k+1}T$, $s_k s_{k+1} T$, $s_{k+1} s_k s_{k+1} T$ are all standard.
Thus, there are two cases to consider. In case (a), we have
$$(\psi_{k+1}\psi_{k}\psi_{k+1} - \psi_{k}\psi_{k+1}\psi_{k})(v\otimes T)=0.$$
In case (b), $i\ne\ell$ and $(\nu_k, \nu_{k+1}, \nu_{k+2}) = (j,\ell,\ell), (\ell,j,\ell), (\ell,\ell,j)$, for some $j\in I$ with $j\ne\ell$.
Using (ii), we have the following.
\begin{itemize}
\item[(1)]
If $(\nu_k, \nu_{k+1}, \nu_{k+2}) = (j,\ell,\ell)$ and $a_{j,\ell}\ne0$ then
$$\psi_k(v\otimes T)=0, \quad \psi_k\psi_{k+1}(v\otimes T)=\psi_k(\partial v\otimes T)=0.$$
\item[(2)]
If $(\nu_k, \nu_{k+1}, \nu_{k+2}) = (j,\ell,\ell)$ and $a_{j,\ell}=0$ then
$$\psi_{k+1}\psi_k\psi_{k+1}(v\otimes T)=\partial v\otimes s_{k+1}s_k T=\psi_k\psi_{k+1}\psi_k(v\otimes T).$$
\item[(3)]
If $(\nu_k, \nu_{k+1}, \nu_{k+2}) = (\ell,j,\ell)$ and $a_{j,\ell}\ne0$ then
$$\psi_k(v\otimes T)=0, \quad \psi_{k+1}(v\otimes T)=0.$$
\item[(4)]
If $(\nu_k, \nu_{k+1}, \nu_{k+2}) = (\ell,j,\ell)$ and $a_{j,\ell}=0$ then
$$\psi_{k+1}\psi_k\psi_{k+1}(v\otimes T)=\partial v\otimes T=\psi_k\psi_{k+1}\psi_k(v\otimes T).$$
\item[(5)]
If $(\nu_k, \nu_{k+1}, \nu_{k+2}) = (\ell,\ell,j)$ and $a_{j,\ell}\ne0$ then
$$\psi_{k+1}(v\otimes T)=0, \quad \psi_{k+1}\psi_k(v\otimes T)=\psi_{k+1}(\partial v\otimes T)=0.$$
\item[(6)]
If $(\nu_k, \nu_{k+1}, \nu_{k+2}) = (\ell,\ell,j)$ and $a_{j,\ell}=0$ then
$$\psi_{k+1}\psi_k\psi_{k+1}(v\otimes T)=\partial v\otimes s_ks_{k+1}T=\psi_k\psi_{k+1}\psi_k(v\otimes T).$$
\end{itemize}
Therefore, we always have $(\psi_{k+1}\psi_{k}\psi_{k+1} - \psi_{k}\psi_{k+1}\psi_{k})e(\nu)(v\otimes T)=0$. Thus, it is enough to prove
that if $\nu_k=\nu_{k+2}$ then
$$ \frac{\mathcal{Q}_{\nu_k,\nu_{k+1}}(x_k,x_{k+1}) - \mathcal{Q}_{\nu_k,\nu_{k+1}}(x_{k+2},x_{k+1})}{x_k-x_{k+2}} (v\otimes T)=0.$$

Suppose that $\nu_k=\nu_{k+2}$. Then (iii) implies that $(\nu_k,\nu_{k+1},\nu_{k+2})=(0,1,0)$ or $(\ell,j,\ell)$ with $a_{\ell,j}=0$.
If $(\nu_k, \nu_{k+1}, \nu_{k+2}) = (0,1,0)$, then
$$\frac{\mathcal{Q}_{0, 1}(x_k,x_{k+1}) - \mathcal{Q}_{0, 1}(x_{k+2},x_{k+1})}{x_k-x_{k+2}}(v\otimes T) = (x_k + x_{k+2}) (v\otimes T) = 0.$$
If $(\nu_k, \nu_{k+1}, \nu_{k+2}) = (\ell, j, \ell)$ with $a_{\ell,j}=0$, then $\mathcal{Q}_{\ell,j}(u,v)=1$ and
$$ \frac{\mathcal{Q}_{\ell,j}(x_k,x_{k+1}) - \mathcal{Q}_{\ell,j}(x_{k+2},x_{k+1})}{x_k-x_{k+2}} (v\otimes T)=0.$$
We have completed the proof.
\end{proof}

\begin{cor} \label{nonsplit seq}
Define
$$\widetilde{\mathcal{L}}_\ell = \bigoplus_{T\in\ST(\lambda(\ell))} \bR (\begin{smallmatrix} 1 \\ 0 \end{smallmatrix})\otimes T.$$

Then $\widetilde{\mathcal{L}}_\ell$ is an $\fqH(2\delta-\alpha_\ell)$-submodule of $\mathcal{L}_\ell$ and we have a nonsplit exact sequence
$$  0 \longrightarrow \widetilde{\mathcal{L}}_\ell \longrightarrow \mathcal{L}_\ell \longrightarrow \widetilde{\mathcal{L}}_\ell \longrightarrow 0. $$
\end{cor}
\begin{proof}
We prove that the exact sequence is nonsplit. The other parts are clear by \eqref{reducible}.
By definition, all $x_k$ act as $0$ on $\widetilde{\mathcal{L}}_\ell$.
Thus, $\mathcal{L}_\ell\not\simeq \widetilde{\mathcal{L}}_\ell \oplus \widetilde{\mathcal{L}}_\ell$
because the action of $x_k$ is nonzero on $\mathcal{L}_\ell$.
\end{proof}

\begin{lemma} \label{Lem: L_i}
\begin{enumerate}
\item The modules $\mathcal{L}_i$, for $0\le i\le\ell-1$, and $\widetilde{\mathcal{L}}_\ell$ are irreducible.
\item
For $0\le i,j \le\ell-1$, we have $ E_j\mathcal{L}_i =0 $ if and only if $j \ne i \pm 1$. Moreover,
$$ E_i\mathcal{L}_j \simeq E_j\mathcal{L}_i $$
and they are irreducible if they are nonzero.
\item $E_{i} \widetilde{\mathcal{L}}_\ell =0$ if and only if $i \ne \ell-1$, and $E_{\ell-1} \widetilde{\mathcal{L}}_\ell$ is irreducible.
\item $ E_{\ell-1} \mathcal{L}_\ell $ is isomorphic to $ E_{\ell} \mathcal{L}_{\ell-1} $, and there is a nonsplit exact sequence:
$$  0 \longrightarrow E_{\ell-1} \widetilde{\mathcal{L}}_\ell \longrightarrow E_{\ell} \mathcal{L}_{\ell-1} \longrightarrow E_{\ell-1} \widetilde{\mathcal{L}}_\ell \longrightarrow 0. $$
\end{enumerate}
\end{lemma}
\begin{proof}
We show that $\mathcal{L}_i$, for $0\le i\le\ell-1$, are irreducible. Let $W$ be a nonzero $\fqH(2\delta-\alpha_i)$-submodule of $\mathcal{L}_i$. Then
$$ W=\bigoplus_{T\in\ST(\lambda(i))} e(\res(T))W $$
and $e(\res(T))W\ne 0$, for some $T\in\ST(\lambda(i))$.
Noting that $\mathcal{N}$ is irreducible as a $NH_2^2$-module, we may assume that $0\ne\mathcal{N}\otimes T\subseteq W$, for some $T\in\ST(\lambda(i))$.
To conclude that $W=\mathcal{L}_i$, it suffices to show the following: for any $T\in\ST(\lambda(i))$,
there is a sequence $T_0,T_1,\dots$ of standard tableaux which starts at $T$ and
ends at the canonical tableau such that $T_{k+1}=s_{i_k}T_k$ for some $i_k$.
Let $S$ be a standard tableau and suppose that $1,2,\dots,a-1$ are on the first row,
$a,a+1,\dots,b$ are on the second row and $b+1$ is on the first row again. Then
we swap $b$ and $b+1$. Repeating the procedure, we reach the stage $a=b$ and when we swap $b$ and $b+1$,
we obtain a standard tableau such that $1,2,\dots,a$ are on the first row and $a+1$ is on the second row.
Hence, repeated use of the procedure starting with $T$ gives the desired sequence of standard tableaux.

In the same manner, we may prove that $\widetilde{\mathcal{L}}_\ell$ and $E_i\mathcal{L}_j$, for $j=i\pm 1$, are irreducible.
In particular, we have proved (1).

Suppose that $0\le i, j\le \ell-1$. By the definition of the module structure $\eqref{Eq: Li}$, we have
$$ \mathrm{ch}_q \mathcal{L}_i = (1+q^2) \sum_{T \in \ST(\lambda(i))} \res(T). $$
Thus, $\mathrm{ch}_q(E_j\mathcal{L}_i)=\mathrm{ch}_q(E_i\mathcal{L}_j)\ne0$ if $j=i\pm 1$ and $\mathrm{ch}_q(E_j\mathcal{L}_i)=0$ otherwise.
Since $E_i\mathcal{L}_j$ are irreducible if they are nonzero, (2) follows. The proof of (3) is similar.

It is clear from the definition of the modules that $E_{\ell-1}\mathcal{L}_\ell\simeq E_\ell\mathcal{L}_{\ell-1}$. Applying the exact functor
$E_{\ell-1}$ to the exact sequence from Corollary \ref{nonsplit seq}, we obtain the exact sequence
$$  0 \longrightarrow E_{\ell-1} \widetilde{\mathcal{L}}_\ell \longrightarrow E_{\ell} \mathcal{L}_{\ell-1} \longrightarrow E_{\ell-1} \widetilde{\mathcal{L}}_\ell \longrightarrow 0. $$
It is nonsplit because $E_{\ell} \mathcal{L}_{\ell-1}$ is not homogeneous while $E_{\ell-1} \widetilde{\mathcal{L}}_\ell$ is homogeneous.
\end{proof}

\begin{Rmk}
The irreducible modules $\mathcal{L}_0$, $\mathcal{L}_1$, $\ldots$, $\mathcal{L}_{\ell-1}$, $\widetilde{\mathcal{L}}_{\ell}$ correspond
to the Young walls $Y_0$, $Y_1$, $\ldots$, $Y_{\ell-1}$, $Y_\ell$ given in $\eqref{Eq: Young walls Yi}$
via the crystal isomorphism from Theorem \ref{Thm: crystals}. Hence,
some parts of Lemma \ref{Lem: L_i} may be proved by using properties $\eqref{Eq: properties of Yi}$ of the Young walls $Y_{i}$.
\end{Rmk}

\begin{prop} \label{Prop: repn type of R(2delta-alpha)}
The irreducible $\fqH(2\delta-\alpha_i)$-module is unique, for all $i\in I$, and we have the following
structure theorem.
\begin{enumerate}
\item For $0\le i\le\ell-1$, we have an algebra isomorphism $\fqH(2\delta-\alpha_i)\simeq \Mat(\dim \mathcal{L}_i, \bR)$,
where $\mathcal{L}_i$ is the unique irreducible $\fqH(2\delta-\alpha_i)$-module defined in $\eqref{Eq: Li}$ and
$$
\dim \mathcal{L}_i=2\binom{h-2}{i}-2\binom{h-2}{i-1}.
$$
\item We have an algebra isomorphism $\fqH(2\delta-\alpha_\ell)\simeq\Mat(\dim \widetilde{\mathcal{L}}_\ell, \bR[x]/(x^2))$,
where $\widetilde{\mathcal{L}}_\ell$ is the unique irreducible $\fqH(2\delta-\alpha_\ell)$-module defined in Corollary \ref{nonsplit seq} and
$$
\dim \widetilde{\mathcal{L}}_\ell=\binom{h-2}{\ell}-\binom{h-2}{\ell-1}.
$$
Furthermore, $\mathcal{L}_\ell$ is the unique indecomposable projective $\fqH(2\delta-\alpha_\ell)$-module.
\end{enumerate}
\end{prop}
\begin{proof}
It follows from Lemma \ref{extremal weights} and Proposition \ref{Prop: repn type} that $\fqH(2\delta-\alpha_i)$ is a simple algebra,
for $0\le i\le\ell-1$. Then $\mathcal{L}_i$ is the unique irreducible $\fqH(2\delta-\alpha_i)$-module by Lemma \ref{Lem: L_i}(1).
Thus, $\fqH(2\delta-\alpha_i)\simeq\End_\bR(\mathcal{L}_i)$ and the hook length formula for the number of elements in $\ST(\lambda(i))$
gives the formula for $\dim \mathcal{L}_i$. Hence (1) follows. To prove (2), observe that
$$r_{\ell-1} \cdots r_{1} r_{0} (\Lambda_0-\delta) = r_{\ell-1} \cdots r_{1} r_{0} (\Lambda_0)-\delta = \Lambda_0-2\delta+\alpha_{\ell}.$$
Then, $\fqH(2\delta-\alpha_\ell)$ and $\fqH(\delta)$ have the same number of simple modules
by Proposition \ref{Prop: repn type}. As $\fqH(\delta)\simeq \bR[x]/(x^2)$, $\widetilde{\mathcal{L}}_\ell$ is the unique
$\fqH(2\delta-\alpha_\ell)$-module by Lemma \ref{Lem: L_i}(1) again. Let $P_\ell$ be the indecomposable projective
$\fqH(2\delta-\alpha_\ell)$-module. Since
$$\dim \fqH(2\delta-\alpha_\ell)=(\dim P_\ell)(\dim \widetilde{\mathcal{L}}_\ell),$$
and Theorem \ref{Thm: dimension formula} shows that $\dim \fqH(2\delta-\alpha_\ell)=2(\dim \widetilde{\mathcal{L}}_\ell)^2$,
$P_\ell$ is the uniserial module of length $2$. As Corollary \ref{nonsplit seq} shows that we have a surjective
homomorphism $P_\ell\to\mathcal{L}_\ell$, we have $P_\ell\simeq\mathcal{L}_\ell$. Hence $\mathcal{L}_\ell$ is the unique
indecomposable projective $\fqH(2\delta-\alpha_\ell)$-module.

To finish the proof, we observe that the regular representation gives an algebra isomorphism
$$\fqH(2\delta-\alpha_\ell)\simeq \End_{\fqH(2\delta-\alpha_\ell)}(\mathcal{L}_\ell^{\oplus \dim\widetilde{\mathcal{L}}_\ell})^{\rm op}
\simeq\Mat(\dim\widetilde{\mathcal{L}}_\ell, \End_{\fqH(2\delta-\alpha_\ell)}(\mathcal{L}_\ell)^{\rm op})$$
and $\End_{\fqH(2\delta-\alpha_\ell)}(\mathcal{L}_\ell)\simeq \bR[x]/(x^2)$. Thus,
$\fqH(2\delta-\alpha_\ell)\simeq\Mat(\dim \widetilde{\mathcal{L}}_\ell, \bR[x]/(x^2))$.
\end{proof}

\subsection{Representations of $\fqH(2\delta)$} \label{Section: repn of R(2delta)}
In this section, we give explicit constructions of various $\fqH(2\delta)$-modules
using results from Section \ref{Section: repn of R(2delta-alpha)}.

\begin{lemma} \label{S for i}
By declaring that $x_{h}$ and $\psi_{h-1}$ act as $0$, and $e(\nu)$, for $\nu\in I^{2\delta}$, as
$$ e(\nu) (v\otimes T) = \left\{
                           \begin{array}{ll}
                             v \otimes T & \hbox{ if } \nu = \res(T) * i, \\
                             0 & \hbox{ otherwise,}
                           \end{array}
                         \right.
$$
where $\res(T) * i$ is the concatenation of $\res(T)$ and $(i)$, the irreducible $\fqH(2\delta-\alpha_i)$-module
$\mathcal{L}_i$ extends to an irreducible $\fqH(2\delta)$-module, for $0\le i\le\ell-1$.
\end{lemma}
\begin{proof}
We may assume that $\nu = \res(T)*i$. Then $\eqref{Eq: lambda i}$ tells
$$ \nu_{h-2} \ne i=\nu_h \quad \text{ and } \quad \nu_{h-1} = i \pm 1\ne \nu_h.  $$
Thus, the relations for
$(\psi_{h-1}\psi_{h-2}\psi_{h-1}-\psi_{h-2}\psi_{h-1}\psi_{h-2})e(\nu)$ and
$(\psi_{h-1}x_l-x_{s_k(l)}\psi_{h-1})e(\nu)$ hold. If $\nu_{h-1}\ne \ell$, then
$a_{\nu_{h-1}, \nu_h}=a_{i\pm 1,i}\ne 0$ and $x_{h-1}(v\otimes T)=0$ imply that
$$ \mathcal{Q}_{\nu_{h-1},\nu_h}(x_{h-1},x_h)(v\otimes T)=0. $$
If $\nu_{h-1}= \ell$, then $a_{\nu_{h-1}, \nu_h}\ne0$ forces $\nu_h=\ell-1$, and
$\mathcal{Q}_{\nu_{h-1},\nu_h}(x_{h-1},x_h)(v\otimes T)$ is equal to
$$ \mathcal{Q}_{\ell,\ell-1}(x_{h-1},x_h)(v\otimes T)=(x_{h-1}^2+x_h)(v\otimes T)=\xi_2^2v\otimes T=0. $$
Thus, the relation for $\psi_{h-1}^2e(\nu)$ holds.
It is easy to check the remaining defining relations.
\end{proof}

\begin{defn}
We denote the irreducible $\fqH(2\delta)$-module defined in Lemma \ref{S for i} by $\mathcal{S}_i$,
for $i=0, 1, \dots, \ell-1$.
\end{defn}

\begin{lemma} \label{S for ell}
By declaring that $\psi_{h-1}$ act as $0$, $x_h$ and $e(\nu)$, for $\nu\in I^{2\delta}$, as
\begin{align*}
x_{h} (v\otimes T) = \xi_2 v \otimes T, \qquad
e(\nu) (v\otimes T) = \left\{
                           \begin{array}{ll}
                             v \otimes T & \hbox{ if } \nu = \res(T) * \ell, \\
                             0 & \hbox{ otherwise,}
                           \end{array}
                         \right.
\end{align*}
$\mathcal{L}_\ell$ extends to a $\fqH(2\delta)$-module. Further, the $\fqH(2\delta-\alpha_\ell)$-submodule
$\widetilde{\mathcal{L}}_\ell$ is stable under the action of $x_h$ and $e(\nu)$, for $\nu\in I^{2\delta}$.
\end{lemma}
\begin{proof}
Let $\nu = \res(T)*\ell$ and we check the defining relations. By $\eqref{Eq: lambda i}$, we have
$$ \nu_{h-2} = \ell-2 \text{ or } \ell \quad \text{ and } \quad \nu_{h-1} = \ell-1\ne \ell=\nu_h.$$
Then, it is easy to check the defining relations except for
$(\psi_{h-1}\psi_{h-2}\psi_{h-1} - \psi_{h-2}\psi_{h-1}\psi_{h-2})e(\nu)$ when $(\nu_{h-2}, \nu_{h-1}, \nu_{h}) = (\ell,\ell-1,\ell)$.
In this case, we have to prove
$$
\frac{\mathcal{Q}_{\ell, \ell-1}(x_{h-2},x_{h-1}) - \mathcal{Q}_{\ell, \ell-1}(x_{h},x_{h-1})}{x_{h-2}-x_{h}}(v\otimes T)=0.
$$
But the left hand side is equal to
$$ (x_{h-2} + x_{h}) (v\otimes T) = (\xi_1+\xi_2)v \otimes T = 0. $$
Hence, we have checked all the defining relations. Finally,
it is clear that $\widetilde{\mathcal{L}}_\ell$ is stable under the action of $x_h$ and $e(\nu)$, for $\nu\in I^{2\delta}$.
\end{proof}

\begin{defn}
We denote the $\fqH(2\delta)$-module defined in Lemma \ref{S for ell} by $\widehat{\mathcal{S}}_\ell$
and the unique irreducible $\fqH(2\delta)$-submodule by $\mathcal{S}_\ell$.
\end{defn}

By construction, $ \varepsilon_j(\mathcal{S}_i) = \delta_{ij}$, for $0\le i,j\le \ell$.
Hence $\mathcal{S}_0, \dots , \mathcal{S}_\ell$ are pairwise non-isomorphic $\fqH(2\delta)$-modules.
Using Theorem \ref{Thm: crystals}, $\eqref{Eq: num of irr R(2delta)-modules}$ and
Corollary \ref{nonsplit seq}, we have the following lemma.

\begin{lemma} \label{Lem: irr of R(2delta)}
\begin{enumerate}
\item The modules $\mathcal{S}_i$ $(i=0,1, \ldots, \ell)$ form a complete list of irreducible $\fqH(2\delta)$-modules.
\item There is a nonsplit exact sequence
$$  0 \longrightarrow \mathcal{S}_\ell  \longrightarrow \widehat{\mathcal{S}}_\ell \longrightarrow \mathcal{S}_\ell  \longrightarrow 0.  $$
\item For $i, j\in I$, we have
$$ E_j \mathcal{S}_i \simeq \left\{
                              \begin{array}{ll}
                                \mathcal{L}_i & \hbox{ if } i=j \ne \ell, \\
                                 \widetilde{\mathcal{L}}_\ell & \hbox{ if } i=j=\ell, \\
                                0 & \hbox{ otherwise.}
                              \end{array}
                            \right.
  $$
\end{enumerate}
\end{lemma}

\bigskip

To determine the basic algebra of $\fqH(2\delta)$ and its representation type, we will need a uniserial module $\mathcal{M}$ with
$\Top(\mathcal{M}) \simeq \mathcal{S}_{\ell-1}$ and $\Rad(\mathcal{M}) \simeq \widehat{\mathcal{S}}_\ell$, which we shall construct here.

\begin{defn}
For each $T \in \ST(\lambda(\ell-1))$ with $\res_{h-1}(T) = \ell $, define $T^\vee\in \ST(\lambda(\ell))$
to be the standard tableau obtained from $T$ by moving the node $(1, \ell+2)$ to the node $(2, \ell+1)$.
\end{defn}

\begin{Ex}
The following is an example when $\ell=3$.

$$
\xy
(-6,0)*{T = };
(0,6)*{};(30,6)*{} **\dir{-};
(0,0)*{};(30,0)*{} **\dir{-};
(6,-6)*{};(18,-6)*{} **\dir{-};
(0,0)*{};(0,6)*{} **\dir{-};
(6,-6)*{};(6,6)*{} **\dir{-};
(12,-6)*{};(12,6)*{} **\dir{-};
(18,-6)*{};(18,6)*{} **\dir{-};
(24,0)*{};(24,6)*{} **\dir{-};
(30,0)*{};(30,6)*{} **\dir{-};
(3,3)*{_1}; (9,3)*{_2}; (15,3)*{_3}; (21,3)*{_5}; (27,3)*{_7};
(9,-3)*{_4}; (15,-3)*{_6};
(33,0)*{,};
(52,0)*{T^\vee = };
(60,6)*{};(84,6)*{} **\dir{-};
(60,0)*{};(84,0)*{} **\dir{-};
(66,-6)*{};(84,-6)*{} **\dir{-};
(60,0)*{};(60,6)*{} **\dir{-};
(66,-6)*{};(66,6)*{} **\dir{-};
(72,-6)*{};(72,6)*{} **\dir{-};
(78,-6)*{};(78,6)*{} **\dir{-};
(84,-6)*{};(84,6)*{} **\dir{-};
(63,3)*{_1}; (69,3)*{_2}; (75,3)*{_3}; (81,3)*{_5};
(69,-3)*{_4}; (75,-3)*{_6}; (81,-3)*{_7};
(87,0)*{.};
\endxy
$$
\end{Ex}

\begin{lemma} \label{construction of M}
We change the action of $\psi_{h-1}$ on the $\fqH(2\delta)$-module $\widehat{\mathcal{S}}_\ell \oplus \mathcal{S}_{\ell-1}$ to
\begin{align*}
\psi_{h-1} (v\otimes T) = \left\{
                            \begin{array}{ll}
                              v \otimes T^\vee & \hbox{ if } \res_{h-1}(T) = \ell,  \\
                              0 & \hbox{ otherwise,}
                            \end{array}
                          \right.
\end{align*}
and keep the action of the other generators unchanged.
Then, we obtain an $\fqH(2\delta)$-module.
\end{lemma}
\begin{proof}
It suffices to check the defining relations which involve $\psi_{h-1}$. It also suffices to check the relations on
$v\otimes T$ with $T\in\ST(\lambda(\ell-1))$. Note that $\res_h(T)=\ell-1$.

Suppose that $\nu\ne\res(T)$. If $e(s_{h-1}\nu)\psi_{h-1}(v\otimes T)\ne0$ then we have
$$ \res_{h-1}(T)=\ell, \quad (\nu_h, \nu_{h-1}) =(\res_{h-1}(T^\vee), \res_h(T^\vee))=(\ell-1, \ell) $$
and $\nu_k=\res_k(T^\vee)=\res_k(T)$, for $1\le k\le h-2$. It contradicts $\nu\ne\res(T)$, and we have
$$ \psi_{h-1}e(\nu)(v\otimes T)=e(s_{h-1}\nu)\psi_{h-1}(v\otimes T). $$
If $\nu=\res(T)$ then explicit computation for the cases $\res_{h-1}(T)\ne\ell$ and $\res_{h-1}(T)=\ell$ proves
$\psi_{h-1}e(\nu)(v\otimes T)=e(s_{h-1}\nu)\psi_{h-1}(v\otimes T)$.

To prove the remaining defining relations, we may assume that $\nu=\res(T)$.

Suppose that $1\le k\le h-3$. Then the location of $h-1$ in $s_k T$ is the
same as the location in $T$. It follows that if $\psi_{h-1}\psi_k(v\otimes T)\ne0$ or $\psi_k\psi_{h-1}(v\otimes T)\ne0$
then $\nu_{h-1}=\ell$ and $s_k T$ is standard. Hence we assume that $\nu_{h-1}=\ell$ and $s_k T$ is standard. Then explicit computation shows
$\psi_{h-1}\psi_k(v\otimes T)=\psi_k\psi_{h-1}(v\otimes T)$.

Next we consider the relation for $(\psi_{h-1} x_k - x_{s_{h-1}(k)} \psi_{h-1})e(\nu)$.
Suppose that $k\ne h-1,h$. If $\psi_{h-1}x_k(v\otimes T)\ne0$ or $x_k\psi_{h-1}(v\otimes T)\ne0$
then $k$ is located at $(1,\ell+1)$ and $h-1$ is located at $(1,\ell+2)$. Hence we assume that
$k$ is located at $(1,\ell+1)$ and $h-1$ is located at $(1,\ell+2)$. Then direct computation shows
$$ (\psi_{h-1} x_k - x_k \psi_{h-1})(v\otimes T)=\xi_1v\otimes T^\vee-\xi_1v\otimes T^\vee=0.$$
If $k=h$ then $ (\psi_{h-1} x_h - x_{h-1} \psi_{h-1})(v\otimes T) = - x_{h-1} \psi_{h-1}(v\otimes T)=0$. Suppose that $k=h-1$.
If $\psi_{h-1} x_{h-1}(v\otimes T)\ne0$ or $x_{h} \psi_{h-1}(v\otimes T)\ne0$ then $h-1$ is located at $(1,\ell+2)$. Thus
we assume that $h-1$ is located at $(1,\ell+2)$. Then, recalling the action of $x_h$ on $\widehat{\mathcal{S}_\ell}$, we have
$$ (\psi_{h-1} x_{h-1} - x_{h} \psi_{h-1})(v\otimes T) = \xi_2v\otimes T^\vee - \xi_2v\otimes T^\vee = 0. $$
Hence the relation for $(\psi_{h-1} x_k - x_{s_{h-1}(k)} \psi_{h-1})e(\nu)$ holds.

Since $\psi_{h-1}$ acts as $0$ on $\widehat{\mathcal{S}_\ell}$, we have $\psi_{h-1}^2 (v \otimes T)=0$. On the other hand,
we already know
$$ \psi_{h-1}^2 (v \otimes T)=\mathcal{Q}_{\nu_{h-1},\nu_h} (x_{h-1}, x_h)(v\otimes T) $$
holds if we replace $\psi_{h-1}$ with $0$, because it is one of the defining relations on $\mathcal{S}_{\ell-1}$.
It implies the relation for $\psi_{h-1}^2e(\nu)$.

Finally, we check the relation for $(\psi_{h-1}\psi_{h-2}\psi_{h-1} - \psi_{h-2}\psi_{h-1}\psi_{h-2})e(\nu)$. Note that $\psi_{h-2}$ acts as $0$
on $\widehat{\mathcal{S}_\ell}$. Hence
$(\psi_{h-1}\psi_{h-2}\psi_{h-1} - \psi_{h-2}\psi_{h-1}\psi_{h-2})(v\otimes T)=0$. On the other hand, we have
$\nu_{h-2}=\ell-2$ or $\ell$, and $\nu_h=\ell-1$. Hence $\nu_{h-2}=\nu_h$ does not occur, and the relation
for $(\psi_{h-1}\psi_{h-2}\psi_{h-1} - \psi_{h-2}\psi_{h-1}\psi_{h-2})e(\nu)$ holds.
\end{proof}

\begin{lemma} \label{Lem: M}
Let $\mathcal{M}$ be the $\fqH(2\delta)$-module constructed in Lemma \ref{construction of M}. Then
the radical series of $\mathcal{M}$ is given as follows:
$$ \mathcal{M} \simeq
                        \begin{array}{ll}
                          \mathcal{S}_{\ell-1} \\
                          \mathcal{S}_{\ell} \\
                          \mathcal{S}_{\ell}
                        \end{array}
 $$
\end{lemma}
\begin{proof}
Since the action of $\psi_{h-1}$ on $\mathcal{M}$ is nonzero, we have a nonsplit exact sequence
$$  0 \longrightarrow \widehat{\mathcal{S}_\ell}  \longrightarrow \mathcal{M} \longrightarrow \mathcal{S}_{\ell-1}  \longrightarrow 0.  $$
Hence $\Soc(\mathcal{M}) \subseteq \widehat{\mathcal{S}_\ell}$, and $\Soc(\mathcal{M})=\mathcal{S}_{\ell}$ by Lemma \ref{Lem: irr of R(2delta)}(2).
If $\mathcal{S}_{\ell}$ appeared in $\Top(\mathcal{M})$, then $\mathcal{M}/\Soc(\mathcal{M})$ would be a semisimple module
$\mathcal{S}_{\ell-1}\oplus \mathcal{S}_\ell$. Then, $\psi_{h-1}$ would act as $0$ on $\mathcal{M}/\Soc(\mathcal{M})$.
Recall the definiton of the action:
\begin{align*}
\psi_{h-1} (v\otimes T) = \left\{
                            \begin{array}{ll}
                              v \otimes T^\vee & \hbox{ if } \res_{h-1}(T) = \ell,  \\
                              0 & \hbox{ otherwise.}
                            \end{array}
                          \right.
\end{align*}
It follows that the action of $\psi_{h-1}$ on $\mathcal{M}/\Soc(\mathcal{M})$ is also nonzero. Thus,
$\Top(\mathcal{M})=\mathcal{S}_{\ell-1}$ and $\Rad(\mathcal{M})=\widehat{\mathcal{S}_\ell}$. Then
Lemma \ref{Lem: irr of R(2delta)}(2) implies that $\Rad^2(\mathcal{M})=\Soc(\mathcal{M})$.
\end{proof}

\section{Representation type} \label{Sec: repn type}
In this section, we give a criterion for representation type of $\fqH(\beta)$.
When $\ell=1$, the algebra $\fqH(\beta)$ is of type $A_1^{(1)}$. Erdmann-Nakano theorem \cite{EN02} and the isomorphism theorem
by Brundan-Kleshchev \cite{BK09} and Rouquier \cite{R08} and \cite{AIP13} for arbitrary parameters tell us that $\fqH(\beta)$ is
\begin{enumerate}
\item simple if $\Lambda_0-\beta \in \weyl \Lambda_0$,
\item not semisimple but of finite type if $\Lambda_0-\beta \in \weyl \Lambda_0-\delta$,
\item of tame type if $\Lambda_0-\beta \in \weyl \Lambda_0-2\delta$,
\item of wild type otherwise.
\end{enumerate}
Thus, an interesting question is to find a criterion for the representation type of $\fqH(\beta)$ when $\ell \ge 2$.
From now on, we assume that $\fqH(\beta)$ is of type $D_{\ell+1}^{(2)}$ for $\ell \ge  2$.

\subsection{The algebra $\fqH(\delta)$} In Section \ref{finite case} we showed that $\fqH(\delta)$ is generated by $x_{\ell+1}$ and
the algebra homomorphism $\bR[x]\to \fqH(\delta)$ which sends $x$ to $x_{\ell+1}$ induces
$\bR[x]/(x^2) \simeq \fqH(\delta)$. Thus, $\fqH(\delta)$ is not semisimple but of finite representation type.

\subsection{The algebra $\fqH(2\delta)$} In this subsection, we show that $\fqH(2\delta)$ is special biserial by computing
the radical series of indecomposable projective $\fqH(2\delta)$-modules.

Recall the $\fqH(2\delta-\alpha_i)$-modules $\mathcal{L}_i$ and $\widetilde{\mathcal{L}}_\ell$ defined in Section
\ref{Section: repn of R(2delta-alpha)}. We define
\begin{align*}
\mathcal{P}_i = F_i \mathcal{L}_i, \;\;\text{for $i=0,1,\dots,\ell$}.
\end{align*}
Proposition \ref{Prop: repn type of R(2delta-alpha)} tells that the $\fqH(2\delta-\alpha_i)$-modules $\mathcal{L}_i$ $(0\le i\le \ell)$ are projective.
As the functor $F_i$ sends projective objects to projective objects, $\mathcal{P}_i$ are projective $\fqH(2\delta)$-modules.

Recall the irreducible $\fqH(2\delta)$-modules $\mathcal{S}_i$ from Lemma \ref{Lem: irr of R(2delta)}.
It follows from Lemma \ref{Lem: L_i}(1), Lemma \ref{Lem: irr of R(2delta)}
and the biadjointness of $E_i$ and $F_i$ that
\begin{equation}
\begin{aligned}
\Hom (\mathcal{S}_j, \mathcal{P}_i) &\simeq \Hom (E_i \mathcal{S}_j, \mathcal{L}_i) \simeq \left\{
                                                                                            \begin{array}{ll}
                                                                                              \bR & \hbox{ if } j=i,  \\
                                                                                              0 & \hbox{ if } j \ne i,
                                                                                            \end{array}
                                                                                          \right. \\
\Hom (\mathcal{P}_i, \mathcal{S}_j) &\simeq \Hom (\mathcal{L}_i, E_i \mathcal{S}_j ) \simeq \left\{
                                                                                            \begin{array}{ll}
                                                                                              \bR & \hbox{ if } j=i,  \\
                                                                                              0 & \hbox{ if } j \ne i.
                                                                                            \end{array}
                                                                                          \right.
\end{aligned}
\end{equation}
Thus, $\mathcal{P}_i$ is indecomposable and $\Top(\mathcal{P}_i) \simeq \Soc(\mathcal{P}_i) \simeq \mathcal{S}_i$. In other words,
$\mathcal{P}_i$ is the projective cover of $\mathcal{S}_i$, for $0\le i\le \ell$, and $\fqH(2\delta)$ is weakly symmetric.

\begin{thm} \label{Thm: radical series}
The radical series of $\mathcal{P}_i$ are given as follows:
\begin{align*}
\xy
(0,0)*{\mathcal{P}_0 \simeq}; (8,0)*{\mathcal{S}_0 }; (14,0)*{\mathcal{S}_1 }; (11,-5)*{\mathcal{S}_{0} }; (11,5)*{\mathcal{S}_{0} }; (17,-2)*{,};
(25,0)*{\mathcal{P}_i\simeq}; (34,0)*{\mathcal{S}_{i-1} }; (43,0)*{\mathcal{S}_{i+1} }; (37.5,5)*{\mathcal{S}_{i} }; (37.5,-5)*{\mathcal{S}_{i} }; (60,0)*{(i\ne \ell-1, \ell),};
(80,0)*{\mathcal{P}_{\ell-1}\simeq}; (91,0)*{\mathcal{S}_{\ell-2} }; (100,2.5)*{\mathcal{S}_{\ell} }; (100,-2.5)*{\mathcal{S}_{\ell} }; (96,8)*{\mathcal{S}_{\ell-1} }; (96,-8)*{\mathcal{S}_{\ell-1} };(104,-2)*{,};
(115,0)*{\mathcal{P}_{\ell}\simeq}; (126,2.5)*{\mathcal{S}_{\ell} }; (126,-2.5)*{\mathcal{S}_{\ell-1} }; (136,-2.5)*{\mathcal{S}_{\ell} }; (136,2.5)*{\mathcal{S}_{\ell-1} };
(131,8)*{\mathcal{S}_{\ell} };(131,-8)*{\mathcal{S}_{\ell} };
\endxy
\end{align*}
\end{thm}
\begin{proof}
We first compute $\dim \Hom (\mathcal{P}_i, \mathcal{P}_j)$ to obtain the composition multiplicities.

Suppose that $i\ne j$.
By the biadjointness of $E_i$ and $F_i$ and \cite[Thm.5.1]{KK11}, we have
$$
\Hom(\mathcal{P}_i, \mathcal{P}_j) \simeq \Hom( E_jF_i \mathcal{L}_i, \mathcal{L}_j) \simeq \Hom( F_i E_j \mathcal{L}_i, \mathcal{L}_j) \simeq \Hom(  E_j \mathcal{L}_i, E_i \mathcal{L}_j),
$$
which yields, by  Lemma \ref{Lem: L_i},
$$ \dim \Hom (\mathcal{P}_i, \mathcal{P}_j) = \left\{
                                                \begin{array}{ll}
                                                  2 & \hbox{ if } (i,j) = (\ell-1, \ell), (\ell, \ell-1), \\
                                                  1 & \hbox{ if } j = i\pm 1 \text{ and } (i,j) \ne (\ell-1, \ell), (\ell, \ell-1), \\
                                                  0 & \hbox{ otherwise.}
                                                \end{array}
                                              \right.
$$

Next we consider the case $i=j$. As $E_i \mathcal{L}_i = 0 $ and $ \langle h_i, \Lambda_0 - 2\delta + \alpha_i \rangle >0$, it follows from Theorem \ref{Thm: Ei Fi} that
$$ E_iF_i \mathcal{L}_i \simeq \mathcal{L}_i^{\oplus \langle h_i, \Lambda_0 - 2\delta + \alpha_i \rangle }. $$
Thus, by Lemma \ref{Lem: L_i}, we have
$$ \dim \Hom(\mathcal{P}_i, \mathcal{P}_i) = \dim \Hom( \mathcal{L}_i, E_iF_i \mathcal{L}_i) = \left\{
                                                                                                 \begin{array}{ll}
                                                                                                 3 & \hbox{ if } i =0, \\
                                                                                                 2 & \hbox{ if } i =1, \dots, \ell-1, \\
                                                                                                 4 & \hbox{ if } i =\ell.
                                                                                                 \end{array}
                                                                                               \right.
 $$
Therefore, we have
\begin{align*}
[\mathcal{P}_0] &= 3[\mathcal{S}_0] + [\mathcal{S}_1], \\
[\mathcal{P}_i] &= [\mathcal{S}_{i-1}] + 2[\mathcal{S}_i] + [\mathcal{S}_{i+1}]\quad (i=1, \ldots, \ell-1), \\
[\mathcal{P}_{\ell-1}] &= [\mathcal{S}_{\ell-2}] + 2[\mathcal{S}_{\ell-1}] + 2[\mathcal{S}_{\ell}], \\
[\mathcal{P}_\ell] &= 2[\mathcal{S}_{\ell-1}] + 4[\mathcal{S}_\ell]
\end{align*}
in the Grothendieck group $K_0(\fqH(2\delta)\text{-mod})$.

Recall that the algebra $\fqH(\beta)$ has an anti-involution which fixes all the defining generators elementwise, and we have the
corresponding duality on the category of $\fqH(\beta)$-modules:
$$ M\mapsto M^\vee=\Hom_\bR(M,\bR). $$
It is straightforward to check that $\mathcal{S}_i$ are self-dual, so that Lemma \ref{Lem: M} implies that
$$
\mathcal{M} \simeq \begin{array}{c}
                   \mathcal{S}_{\ell-1} \\
                   \mathcal{S}_{\ell} \\
                   \mathcal{S}_{\ell}
                 \end{array}
\quad \text{ and } \quad
\mathcal{M}^\vee \simeq \begin{array}{c}
                   \mathcal{S}_{\ell} \\
                   \mathcal{S}_{\ell} \\
                   \mathcal{S}_{\ell-1}.
                 \end{array}
$$
Since $\mathcal{P}_i$ is the projective cover of $\mathcal{S}_i$, there are nontrivial surjective homomorphisms
$$ \mathcal{P}_{\ell-1} \twoheadrightarrow \mathcal{M} \quad \text{ and } \quad \mathcal{P}_{\ell} \twoheadrightarrow \mathcal{M}^\vee. $$
As the algebra $\fqH(2\delta)$ is weakly symmetric, the indecomposable projective module $\mathcal{P}_i$
is isomorphic to its dual. Therefore, the heart of $\mathcal{P}_i$ is self-dual, which yields the assertion.
\end{proof}

Recall that a finite dimensional algebra is a {\it special biserial algebra} if the quiver and the relations of its basic algebra
satisfy the following conditions.
\begin{itemize}
\item[(a1)]
For each vertex $i$, the number of incoming arrows is at most $2$.
\item[(a2)]
For each vertex $i$, the number of outgoing arrows is at most $2$.
\item[(b1)]
For each arrow $\alpha$, there is at most one arrow $\beta$ such that $\alpha\beta\ne0$.
\item[(b2)]
For each arrow $\alpha$, there is at most one arrow $\beta$ such that $\beta\alpha\ne0$.
\end{itemize}

\noindent
The notion of special biserial algebras was originally introduced by Skowro\'nski and Waschb\"usch \cite{SW83} in order
to characterize representation-finite algebras which has at most $2$ nonprojective indecomposable direct summands in the middle term of
any almost split sequence.

If the relations of a special biserial algebra are given by monomials in arrows, the algebra is called a {\it string algebra}, and complete
classification of indecomposable modules is known \cite{BR87}. The indecomposable modules are given by
string modules and band modules, where strings and bands are certain walks on the double of the quiver. As we will need some computation of
string modules and the growth rate for the number of bands, we quickly review necessary materials.

A walk on a quiver is a {\it closed walk} if it returns to the initial vertex of the walk. We say that closed walks
$w$ and $w'$ are {\it cyclically equivalent} if $w=w_1\cdots w_n$, where $w_i$ are arrows of the quiver, then
$w'=w_iw_{i+1}\cdots w_nw_1\cdots w_{i-1}$, for some $i$.

\begin{defn}
Let $\bR Q/I$ be a string algebra. For each arrow $\alpha$ in $Q$, we add a new arrow $\alpha^{-1}$ with the opposite direction, and
we denote the resulting quiver by $\widehat{Q}$.

\begin{itemize}
\item[(1)]
A walk $w=w_1\cdots w_n$ $(n\ge0)$ on $\widehat{Q}$ is a {\it string} if $w_{i+1}\ne w_i^{-1}$, for all $i$, and if a subword
$w_iw_{i+1}\cdots w_j$ is a walk on $Q$ or the inverse of a walk on $Q$, then the walk on $Q$ does not vanish in $\bR Q/I$.
We say that $w$ and $w^{-1}=w_n^{-1}\cdots w_1^{-1}$ are equivalent strings.
\item[(2)]
A closed walk $w=w_1\cdots w_n$ $(n\ge1)$ on $\widehat{Q}$ is a {\it band} if all the powers
$w^m$ $(m\ge1)$ are strings and $w$ itself is not cyclically equivalent to a power of a closed walk.
If $w'$ is cyclically equivalent to $w$ or $w^{-1}$, we say that $w$ and $w'$ are equivalent bands.
\end{itemize}
\end{defn}

\begin{Rmk}
It is more common to call the equivalence classes strings and bands.
\end{Rmk}

For a string $w=w_1\cdots w_n$ $(n\ge0)$, we associate an updown diagram. For example, if $w_1, w_2, w_3^{-1}$ and $w_n$ are arrows in $Q$, then

\bigskip
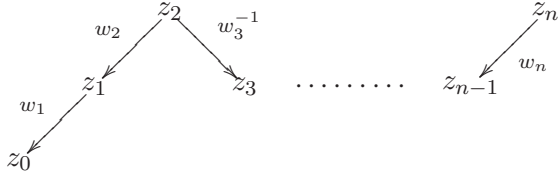
\hspace{1.5cm}
\begin{xy}
(10,-10) *{z_0}="A", (20,0) *{z_1}="B", (54,0) *{\cdots\cdots\cdots},
(30,10) *{z_2}="C", (40,0) *{z_3}="D", (70,0) * {z_{n-1}}="E", (80,10) *{z_n}="F",

\ar  "B";"A"_{w_1}
\ar  "C";"B"_{w_2}
\ar  "C";"D"^{w_3^{-1}}
\ar  "F";"E"^{w_n}
\end{xy}

\bigskip
\noindent
The updown diagram defines the string module $M(w)=\bR z_0\oplus\cdots\oplus\bR z_n$.
More precisely, the action of an arrow $\alpha$ in $Q$ is defined as follows.

\begin{itemize}
\item[(a)]
If $w_i$ is an arrow in $Q$ then $\alpha z_i=\delta_{\alpha, w_i}z_{i-1}$.
\item[(b)]
If $w_i^{-1}$ is an arrow in $Q$ then $\alpha z_{i-1}=\delta_{\alpha, w_i^{-1}}z_i$.
\end{itemize}
The action of vertices on $z_i$ is determined by the end point of $w_i$, which is the same as the initial point of $w_{i+1}$,
in the obvious manner.

Now, we begin by the following corollary of Theorem \ref{Thm: radical series}.

\begin{cor} \label{special biserial}
The quiver of $\fqH(2\delta)$ is given as follows. Moreover, $\fqH(2\delta)$ is a symmetric special biserial algebra.

\bigskip
\hspace{4mm}
\begin{xy}
(10,0) *{\circ}="A", (30,0) *{\circ}="B", (50,0) *{\circ}="C", (90,0) *{\circ}="D", (110,0) *{\circ}="E",
(70,0) *{\cdots\cdots\cdots\cdots},

\ar @(lu,ld) "A";"A"_{\gamma}
\ar @/^/ "A";"B"^{\alpha_1}
\ar @/^/ "B";"A"^{\beta_1}
\ar @/^/ "B";"C"^{\alpha_2}
\ar @/^/ "C";"B"^{\beta_2}

\ar @/^/ "D";"E"^{\alpha_\ell}
\ar @/^/ "E";"D"^{\beta_\ell}
\ar @(rd,ru) "E";"E"_{\delta}
\end{xy}
\end{cor}
\begin{proof}
It follows from Theorem \ref{Thm: radical series} that
$$ \dim \Hom(\mathcal{P}_i, \Rad(\mathcal{P}_j) / \Rad^2(\mathcal{P}_j) )  = \left\{
                                                                               \begin{array}{ll}
                                                                                 1 & \hbox{ if } j = i \pm 1, \hbox{ or } i=j=0,\ell, \\
                                                                                 0 & \hbox{ otherwise. }
                                                                               \end{array}
                                                                             \right.
 $$
We fix lifts of $\Hom(\mathcal{P}_i, \Rad(\mathcal{P}_j) / \Rad^2(\mathcal{P}_j) )$ to $\Hom(\mathcal{P}_i, \mathcal{P}_j )$ as follows.
\begin{align*}
\alpha_i  &= \text{a lift of } \mathcal{P}_{i-1}\twoheadrightarrow \mathcal{S}_{i-1}\subseteq \Rad(\mathcal{P}_{i}) / \Rad^2(\mathcal{P}_{i}), \\
\beta_i  &= \text{a lift of } \mathcal{P}_{i}\twoheadrightarrow \mathcal{S}_{i}\subseteq \Rad(\mathcal{P}_{i-1}) / \Rad^2(\mathcal{P}_{i-1}), \\
\gamma  &= \text{a lift of } \mathcal{P}_{0}\twoheadrightarrow \mathcal{S}_{0}\subseteq \Rad(\mathcal{P}_{0}) / \Rad^2(\mathcal{P}_{0}), \\
\delta  &= \text{a lift of } \mathcal{P}_{\ell}\twoheadrightarrow \mathcal{S}_{\ell}\subseteq\Rad(\mathcal{P}_{\ell}) / \Rad^2(\mathcal{P}_{\ell}).
\end{align*}
Then, $\alpha_i$, $\beta_i$, $\gamma$ and $\delta$ give the quiver of the basic algebra of $\fqH(2\delta)$, and
the relations are
\begin{gather*}
\gamma\alpha_1=0,\;\;\beta_1\gamma=0,\;\;\gamma^2=\alpha_1\beta_1,\quad
\alpha_\ell\beta_\ell=0, \;\; \delta^2=0, \\
\beta_i\alpha_i=\alpha_{i+1}\beta_{i+1},\:\:\text{for $1\le i\le \ell-2$,} \\
\alpha_i\alpha_{i+1}=0,\;\;
\beta_{i+1}\beta_i=0, \;\;\text{for $1\le i\le \ell-1$,}\\
\beta_{\ell-1}\alpha_{\ell-1}=\alpha_\ell\delta\beta_\ell,\;\;
\delta\beta_\ell\alpha_\ell=\beta_\ell\alpha_\ell\delta.
\end{gather*}

Indeed, considering the configuration of the radical series in Theorem \ref{Thm: radical series},
$\gamma\alpha_1=0$, $\beta_1\gamma=0$ and $\alpha_i\alpha_{i+1}=0$, $\beta_{i+1}\beta_i=0$, for $1\le i\le \ell-1$, are clear.
As $F_\ell\widetilde{\mathcal{L}}_\ell$ has composition factors
$$
[F_\ell\widetilde{\mathcal{L}}_\ell]=2[\mathcal{S}_\ell]+[\mathcal{S}_{\ell-1}]
$$
and $\mathcal{P}_\ell=F_\ell\mathcal{L}_\ell$ implies that
$\Top(F_\ell\widetilde{\mathcal{L}}_\ell)\simeq \Top(\mathcal{P}_\ell)$ and $\Soc(F_\ell\widetilde{\mathcal{L}}_\ell)\simeq \Soc(\mathcal{P}_\ell)$,
$$  F_\ell\widetilde{\mathcal{L}}_\ell\; \simeq
                        \begin{array}{ll}
                          \mathcal{S}_{\ell} \\
                          \mathcal{S}_{\ell-1} \\
                          \mathcal{S}_{\ell}
                        \end{array}
 $$
and we have an $\fqH(2\delta)$-submodule of $\mathcal{P}_\ell$ which is isomorphic to $F_\ell\widetilde{\mathcal{L}}_\ell$. It follows that
$$ \Im(\delta)=\Ker(\delta)\simeq F_\ell\widetilde{\mathcal{L}}_\ell $$
and we have $\delta^2=0$. Similarly, we have $\Im(\alpha_\ell)=\Ker(\beta_\ell)\simeq \mathcal{M}$ and
$\Im(\beta_\ell)\simeq \mathcal{M}^\vee$, which implies $\alpha_\ell\beta_\ell=0$. Further, explicit computation shows
\begin{align*}
\Im(\gamma^2)&=\Soc(\mathcal{P}_0)=\Im(\alpha_1\beta_1),\\
\Im(\beta_i\alpha_i)&=\Soc(\mathcal{P}_i)=\Im(\alpha_{i+1}\beta_{i+1}), \hbox{ for $1\le i\le \ell-2$,}\\
\Im(\beta_{\ell-1}\alpha_{\ell-1})&=\Soc(\mathcal{P}_{\ell-1})=\Im(\alpha_\ell\delta\beta_\ell).
\end{align*}
Therefore, $\gamma^2$, $\beta_i\alpha_i$ and $\beta_{\ell-1}\alpha_{\ell-1}$ are nonzero scalar multiple of
$\alpha_1\beta_1$, $\alpha_{i+1}\beta_{i+1}$ and $\alpha_\ell\delta\beta_\ell$, respectively. By adjusting $\alpha_i$'s by nonzero scalar multiples, we
may assume that $\gamma^2=\alpha_1\beta_1$, $\beta_i\alpha_i=\alpha_{i+1}\beta_{i+1}$ and $\beta_{\ell-1}\alpha_{\ell-1}=\alpha_\ell\delta\beta_\ell$ hold.
To show that $\delta\beta_\ell\alpha_\ell=\beta_\ell\alpha_\ell\delta$, it is enough to find that $\End(\mathcal{P}_\ell)$ is commutative.
Let $\nu_0 = (0,1,2, \ldots, \ell) \in I^{\ell+1}$ and $\nu=\nu_0 * \nu_0$. As $e(\nu)\mathcal{S}_\ell\ne0$, we have a surjective map
$\fqH(2\delta)e(\nu)\to \mathcal{P}_\ell$. Thus, $ \mathcal{P}_\ell$ is a direct summand of $\fqH(2\delta)e(\nu)$. Then,
(\ref{endomorphism ring}) below shows that $\End(\fqH(2\delta)e(\nu))$ is a commutative local algebra, and it follows that
$\mathcal{P}_\ell\simeq \fqH(2\delta)e(\nu)$ and $\End(\mathcal{P}_\ell)$ is commutative. Hence, all the relations listed above hold, and
the dimension consideration proves that they give defining relations of the basic algebra of $\fqH(2\delta)$. In particular,
the quiver of $\fqH(2\delta)$ is as claimed, and $\fqH(2\delta)$ is a special biserial algebra.

Next, we construct a trace map on the basic algebra. The basic algebra is
$(4\ell+7)$-dimensional and the following elements in $B_1,\dots,B_{2\ell+1}$ altogether form a basis of the basic algebra, which we denote by $B$:
\begin{itemize}
\item[(1)]
$B_1=\{e_0, \gamma,\gamma^2\}$.
\item[(2)]
$B_{i+1}=\{\alpha_i, \beta_i\}$, for $1\le i\le \ell-1$.
\item[(3)]
$B_{i+\ell}=\{e_i, \alpha_{i+1}\beta_{i+1}\}$, for $1\le i\le \ell-2$.
\item[(4)]
$B_{2\ell-1}=\{e_{\ell-1}, \alpha_\ell\delta\beta_\ell\}$.
\item[(5)]
$B_{2\ell}=\{\alpha_\ell, \beta_\ell, \alpha_\ell\delta, \delta\beta_\ell\}$.
\item[(6)]
$B_{2\ell+1}=\{e_\ell, \delta, \beta_\ell\alpha_\ell, \delta\beta_\ell\alpha_\ell\}$.
\end{itemize}
The basis $B$ has the property that either $b_1b_2=0$ or $b_1b_2\in B$, for $b_1, b_2\in B$.
We define the trace map by values on $B$, and we declare that the nonzero values are
$$
\Tr(\gamma^2)=\Tr(\alpha_2\beta_2)=\cdots=\Tr(\alpha_{\ell-1}\beta_{\ell-1})
=\Tr(\alpha_\ell\delta\beta_\ell)=\Tr(\delta\beta_\ell\alpha_\ell)=1.
$$
Note that $\Tr(b)\ne0$, for $b\in B$, only when the source and the sink of $b$ coincide.
We consider the matrix $(\Tr(b_1b_2))_{b_1,b_2\in B}$. Our task is to show that
it is a nonsingular symmetric matrix. Suppose that the source of $b_1$ is different from the sink of $b_2$. Then
$\Tr(b_1b_2)=0$ and $b_2b_1=0$. If the source of $b_1$ equals the sink of $b_2$,
$\{b_1, b_2\}\subseteq B_i$, for some $i$. Therefore,
it is enough to show that the submatrix $(\Tr(b_1b_2))_{b_1,b_2\in B_i}$ is a nonsingular symmetric matrix.
But it can be checked by direct computation.
\end{proof}

We remark that there exists an example of weakly symmetric special biserial algebras which is not symmetric. See \cite[IV.2.8]{SY11}.

\begin{cor} \label{Cor: fqH(2delta) repn type}
The algebra $\fqH(2\delta)$ is of tame type.
\end{cor}
\begin{proof}
Since biserial algebras have tame or finite representation type \cite{CB95}
(for special biserial algebras it was already proved in \cite{WW85}), we
show that $\fqH(2\delta)$ is not of finite type.
But if we set $x=\delta$ and $y=\beta_\ell\alpha_\ell$, where $\delta, \alpha_\ell, \beta_\ell$
are as in Corollary \ref{special biserial}, we have $\End (\mathcal{P}_\ell)= \bR[x,y]/ (x^2, y^2)$.
The Kronecker algebra $\bR[x,y]/ (x^2, y^2)$ is of tame type, so that $\fqH(2\delta)$ is not of finite type.
\end{proof}

%
%

Based on \cite{DS87}, \cite{S89} and \cite{PS91}, Erdmann and Skowro\'nski has obtained detailed description
of the Auslander-Reiten quiver of a special biserial self-injective algebra \cite[Thm.2.1, Thm.2.2]{ES92}.
In particular, \cite[Thm.2.2]{ES92} says that if there are infinitely many equivalence classes of bands then its stable Auslander-Reiten quiver
has finitely many nonhomogeneous tubes $\Z A_\infty/\langle \tau^p\rangle$,
infinitely many homogeneous tubes, and infinitely many components of the form $\Z A^\infty_\infty$, where $\tau$ is the Auslander-Reiten translation.
We may give more explicit result for $\fqH(2\delta)$ in type $D^{(2)}_{\ell+1}$ ($\ell \ge 2$), as follows.

\begin{prop} \label{AR components}
The stable Auslander-Reiten quiver of $\fqH(2\delta)$ has
\begin{itemize}
\item[(i)]
the unique nonhomogeneous tube $\Z A_\infty/\langle \tau^{2\ell+1}\rangle$ consisting of string modules,
\item[(ii)]
one homogeneous tube consisting of string modules,
\item[(iii)]
infinitely many homogeneous tubes consisting of band modules,
\item[(iv)]
infinitely many components of the form $\Z A^\infty_\infty$ consisting of string modules.
\end{itemize}
\end{prop}
\begin{proof}
Let $A=\bR Q/I$ be the basic algebra of $\fqH(2\delta)$ given in Corollary \ref{special biserial} and define
$B=A/\Soc(A)$. Then $B$ is the string algebra defined by the quiver $Q$ and the relations
\begin{align*}
\gamma^2=0, \;\;\gamma\alpha_1&=0,\;\; \beta_1\gamma=0, \\
\alpha_i\beta_i&=0,\:\:\text{for $1\le i\le \ell$,} \\
\beta_i\alpha_i&=0,\;\;\text{for $1\le i\le \ell-1$,} \\
\alpha_i\alpha_{i+1}&=0,\;\;\text{for $1\le i\le \ell-1$,} \\
\beta_{i+1}\beta_i&=0, \;\;\text{for $1\le i\le \ell-1$,}\\
\delta\beta_\ell\alpha_\ell=\beta_\ell\alpha_\ell\delta&=0,\;\;\alpha_\ell\delta\beta_\ell=0,\;\; \delta^2=0.
\end{align*}

First we show that $B$ has infinitely many equivalence classes of bands. Then it implies that $A$ is not of polynomial growth by the same argument
as in \cite[Lem.1]{S87}, and we are in the case of \cite[Thm.2.2]{ES92}.

We define closed walks $a$ and $b$ as follows. Note that $\beta_\ell\alpha_\ell\ne0$ in $B$.
\begin{align*}
a=\delta\alpha_\ell^{-1}\beta_\ell^{-1},\;\;
b&=\delta\alpha_\ell^{-1}\beta_{\ell-1}\cdots\alpha_2^{-1}\beta_1\gamma^{-1}\alpha_1\beta_2^{-1}\cdots\alpha_{\ell-1}\beta_\ell^{-1}
\quad &\text{if $\ell$ is even.}\\
a=\delta^{-1}\beta_\ell\alpha_\ell,\;\;
b&=\delta^{-1}\beta_\ell\alpha_{\ell-1}^{-1}\cdots\alpha_2^{-1}\beta_1\gamma^{-1}\alpha_1\beta_2^{-1}\cdots\alpha_\ell
\quad &\text{if $\ell$ is odd.}
\end{align*}

We claim that $2^q-2$ closed walks $\{ x_1\cdots x_q \mid x_i=a\; \text{or}\; b \}\setminus\{ a^q, b^q\}$, for a prime $q$, give
$(2^q-2)/q$ equivalence classes of bands.
It is clear that their powers are strings. Thus, it suffices to show that it is not cyclically equivalent to a power of its subword.
Suppose that it is the case. Then, we may assume that the subword starts with $\delta$ or $\delta^{-1}$. But $\delta^\pm$ appear
only as the first alphabet of $a$ and $b$, so that the subword is of the form $x_ix_{i+1}\cdots x_{i+r-1}$ and $r$ divides $q$.
Thus, $r=1$ or $r=q$ because $q$ is a prime, and $r=1$ does not occur because we have excluded $a^q$ and $b^q$. We have proved that
$x_1\cdots x_q$ cannot be cyclically equivalent to a power of its subword, and we have the claim.
In particular, we have constructed infinitely many equivalence classes of bands.
Since all the nonprojective indecomposable $A$-modules are indecomposable $B$-modules and
the stable Auslander-Reiten quiver of $A$ coincides with the Auslander-Reiten quiver of $B$, we consider
the latter. Then, by the general result \cite[Thm.2.2]{ES92}, it only remains to determine the period of
string $B$-modules in components of the form $\Z A_\infty/\langle \tau^p \rangle$, for $p\ge1$,
since band $B$-modules belong to homogeneous tubes \cite[p.165]{BR87}. But string modules on the boundary of such components are known \cite[p.170]{BR87}.
They are $\{ Be_i/B\alpha \mid \text{$i$ is the end point of an arrow $\alpha$}\}$.
Thus, it suffices to determine the period of these string modules with respect to $\tau=D\Tr$.

First of all, $Be_\ell/B\delta$ is exceptional and it has period $1$. Indeed,
$$
Be_\ell\to Be_\ell\twoheadrightarrow
Be_\ell/B\delta=\frac{\langle e_\ell, \delta, \alpha_\ell, \beta_\ell\alpha_\ell, \alpha_\ell\delta\rangle_\bR}{\langle \delta,\alpha_\ell\delta\rangle_\bR}
=M(\beta_\ell\alpha_\ell),
$$
where $Be_\ell\to Be_\ell$ is given by $x\mapsto x\delta$, is the projective resolution, and direct computation shows
$\tau(M(\beta_\ell\alpha_\ell))\simeq M(\beta_\ell\alpha_\ell)$.

Similarly, we compute the $\tau$-orbit through $Be_0/B\gamma=M(\beta_1)$.
Suppose that $\ell$ is even. The first almost split sequence to consider is
$$
0 \to M(\alpha_1) \to M(\beta_1\gamma^{-1}\alpha_1)\to M(\beta_1) \to 0
$$
and we have $\tau(Be_0/B\gamma)=Be_1/B\beta_2$.
The rule to construct almost split sequences for the modules of the form $Be_i/B\alpha$ is that
if $Be_i/B\alpha=M(u)$ and $\tau(Be_i/B\alpha)=M(v)$, then we have the almost split sequence
$$ 0 \to M(v) \to M(u\alpha^{-1}v)\to M(u) \to 0. $$

\noindent
We proceed further as follows.
\begin{align*}
\tau(Be_0/B\gamma)&=M(\alpha_1)=Be_1/B\beta_2, \\
\cdots & \cdots \\
\tau^{\frac{\ell}{2}}(Be_0/B\gamma)&=M(\alpha_{\ell-1})=Be_{\ell-1}/B\beta_\ell, \\
\tau^{\frac{\ell}{2}+1}(Be_0/B\gamma)&=M(\delta\beta_\ell)=Be_{\ell-1}/B\alpha_{\ell-1}, \\
\tau^{\frac{\ell}{2}+2}(Be_0/B\gamma)&=M(\beta_{\ell-2})=Be_{\ell-3}/B\alpha_{\ell-3}, \\
\cdots & \cdots \\
\tau^{\ell}(Be_0/B\gamma)&=M(\beta_2)=Be_1/B\alpha_1.
\end{align*}

We continue the computation and obtain
\begin{align*}
\tau^{\ell+1}(Be_0/B\gamma)&=M(\gamma)=Be_0/B\beta_1, \\
\tau^{\ell+2}(Be_0/B\gamma)&=M(\alpha_2)=Be_2/B\beta_3, \\
\cdots & \cdots \\
\tau^{\frac{3}{2}\ell}(Be_0/B\gamma)&=M(\alpha_{\ell-2})=Be_{\ell-2}/B\beta_{\ell-1},\\
\tau^{\frac{3}{2}\ell+1}(Be_0/B\gamma)&=M(\alpha_\ell\delta)=Be_\ell/B\alpha_\ell, \\
\tau^{\frac{3}{2}\ell+2}(Be_0/B\gamma)&=M(\beta_{\ell-1})=Be_{\ell-2}/B\alpha_{\ell-2}, \\
\cdots & \cdots \\
\tau^{2\ell}(Be_0/B\gamma)&=M(\beta_3)=Be_2/B\alpha_2, \\
\tau^{2\ell+1}(Be_0/B\gamma)&=M(\beta_1)=Be_0/B\gamma.
\end{align*}
Therefore, the period of $Be_0/B\gamma$ is $2\ell+1$, and the $\tau$-orbit contains all the string modules
of the form $Be_i/B\alpha$ but $Be_\ell/B\delta$. Suppose that $\ell$ is odd. Then the computation is entirely similar
and we reach the same conclusion. We have proved that
there exists unique nonhomogeneous tube $\Z A_\infty/\langle \tau^{2\ell+1}\rangle$ consisting of string modules,
and the unique homogeneous tube consisting of string modules.
\end{proof}

By the main theorem given below, Proposition \ref{AR components} gives
the shape of stable Auslander-Reiten quivers for finite quiver Hecke algebras of tame type in type $D^{(2)}_{\ell+1}$.

\subsection{ The algebra $\fqH(3\delta)$} In this subsection, we show that $\fqH(3\delta)$ is of wild type.

Let $\nu_0 = (0,1,2, \ldots, \ell) \in I^{\ell+1}$ and, for $r = 0,1,2$ and $0\le s\le\ell$, or $r=3$ and $s=0$, set
\begin{align*}
\nu^{r,s} &= \underbrace{\nu_0 * \cdots * \nu_0}_{r} * (0,1,2, \ldots, s-1), \\
\beta^{r,s} &= r \delta + \alpha_0 + \alpha_1 + \cdots + \alpha_{s-1},
\end{align*}
where $\nu*\nu'$ is the concatenation of $\nu$ and $\nu'$.
Note that $\nu^{r,s} \in I^{\beta^{r,s}}$. Using the residue pattern $\eqref{Eq: residue pattern}$,
one can show that there exists only one standard tableau $T$ such that $\res(T) = \nu^{r,s} $.
Thus, by Theorem \ref{Thm: dimension formula}, we have
\begin{align} \label{Eq: dim of fqH(beta r s)}
\dim e(\nu^{r,s}) \fqH( \beta^{r,s} )  e(\nu^{r,s}) = 2^r.
\end{align}

\begin{prop} \label{Prop: fqH(3delta)}
The algebra $e(\nu^{3,0}) \fqH(3\delta) e(\nu^{3,0})$ is isomorphic to the quotient algebra of $\bR[x,y,z]$ by the ideal generated by
$x^2$, $y^2 - a xy$ and $z^2 - b xy - c yz - d yz$ for some $a,b,c,d \in \bR$.
\end{prop}
\begin{proof}
We set $e^{r,s} = e(\nu^{r,s}) $, for $r = 0,1,2$ and $0\le s\le \ell$, or $r=3$ and $s=0$.
By a direct computation, we have
$$ \langle h_s, \Lambda_0 - \beta^{r,s} \rangle = \left\{
                                                                                    \begin{array}{ll}
                                                                                      1 & \hbox{ if } s = 0,\dots, \ell-1, \\
                                                                                      2 & \hbox{ if } s = \ell.
                                                                                    \end{array}
                                                                                  \right.
 $$
Thus, Theorem \ref{Thm: Ei Fi} tells that there are $( \fqH(\beta^{r,s}), \fqH(\beta^{r,s}) )$-bimodule monomorphisms
\begin{align*}
\fqH(\beta^{r,s}) &\hookrightarrow e( \beta^{r,s}, s ) \fqH(\beta^{r,s+1}) e(\beta^{r,s}, s)\quad(0\le s\le \ell-1), \\
\bigoplus_{k=0}^1 \fqH(\beta^{r,\ell}) \otimes t^k &\hookrightarrow e( \beta^{r,\ell}, \ell ) \fqH(\beta^{r+1,0}) e(\beta^{r,\ell}, \ell)\quad(s=\ell),
\end{align*}
which yield $( e^{r,s}\fqH(\beta^{r,s})e^{r,s},\ e^{r,s} \fqH(\beta^{r,s}) e^{r,s} )$-bimodule monomorphisms
\begin{equation}  \label{Eq: fqH(3delta) embedding}
\begin{aligned}
e^{r,s} \fqH(\beta^{r,s}) e^{r,s} & \hookrightarrow  e^{r,s+1}  \fqH(\beta^{r,s+1}) e^{r,s+1}\quad(0\le s\le \ell-1),\\
\bigoplus_{k=0}^1 \left( e^{r,\ell} \fqH(\beta^{r,\ell}) e^{r,\ell} \right) \otimes t^k & \hookrightarrow  e^{r+1,0}  \fqH(\beta^{r+1,0}) e^{r+1,0}
\quad(s=\ell).
\end{aligned}
\end{equation}

We first consider the algebra $e^{1,0} \fqH(\beta^{1,0}) e^{1,0}$. Since $\fqH(\beta^{1,0})=\fqH(\delta)$ and $e(\nu^{1,0})$ is the unit element of $\fqH(\delta)$, we have
$$ e^{1,0} \fqH(\beta^{1,0}) e^{1,0} = \fqH(\delta) = \bR[x]/(x^2),\;\;\text{for $x=x_{\ell+1}e^{1,0}$}. $$
Since $ \dim e^{1,s} \fqH( \beta^{1,s} )  e^{1,s} = 2 $ by $\eqref{Eq: dim of fqH(beta r s)}$, the first monomorphism of $\eqref{Eq: fqH(3delta) embedding}$ gives
$$ e^{1,0} \fqH(\beta^{1,0}) e^{1,0} \simeq e^{1,1} \fqH(\beta^{1,1}) e^{1,1} \simeq \cdots \simeq e^{1,\ell} \fqH(\beta^{1,\ell}) e^{1,\ell}, $$
and it follows that
\begin{align} \label{Eq: isom 1}
e^{1,\ell} \fqH(\beta^{1,\ell}) e^{1,\ell} = \bR[x]/(x^2),\;\;\text{for $x=x_{\ell+1}e^{1,\ell}$}.
\end{align}
Let us consider the second monomorphism of $\eqref{Eq: fqH(3delta) embedding}$, for $r=1$. It is
\begin{align} \label{Eq: embedding 1}
\bigoplus_{k=0}^1 \left( e^{1,\ell} \fqH(\beta^{1,\ell}) e^{1,\ell} \right) \otimes t^k  \hookrightarrow  e^{2,0}  \fqH(\beta^{2,0}) e^{2,0},
\end{align}
where $f \otimes t^k$, for $f \in e^{1,\ell} \fqH(\beta^{1,\ell}) e^{1,\ell} $, maps to $f x_{2\ell+2}^k e^{2,0} $. Note that
it is not only a bimodule homomorphism but an algebra homomorpphism
since $x_{\ell+1}$ commutes with $x_{2\ell+2}$ and the monomorphism is induced by the algebra homomorphism
$$
e^{1,\ell} \fqH(\beta^{1,\ell}) e^{1,\ell}\otimes \bR[t]  \longrightarrow  e^{2,0}  \fqH(\beta^{2,0}) e^{2,0}.
$$
It follows from
$$  2 \dim e^{1,\ell} \fqH( \beta^{1,\ell} )  e^{1,\ell} = 4 = \dim e^{2,0} \fqH( \beta^{2,0} )  e^{2,0}$$
that the embedding $\eqref{Eq: embedding 1}$ is an isomorphism of algebras.
Therefore, using $\eqref{Eq: isom 1}$ and the fact that $e^{2,0} \fqH( \beta^{2,0} )  e^{2,0}$ is graded, we can conclude
\begin{align}\label{endomorphism ring}
e^{2,0} \fqH(\beta^{2,0}) e^{2,0} = \bR[x,y]/(x^2, y^2-axy)
\end{align}
for some $a\in \bR$, where $x=x_{\ell+1}e^{2,0}$ and $y=x_{2\ell+2}e^{2,0}$.

In the same manner, as $ \dim e^{2,s} \fqH( \beta^{2,s} )  e^{2,s} = 4 $ by $\eqref{Eq: dim of fqH(beta r s)}$, the first monomorphism of $\eqref{Eq: fqH(3delta) embedding}$ gives
$$ e^{2,0} \fqH(\beta^{2,0}) e^{2,0} \simeq e^{2,1} \fqH(\beta^{2,1}) e^{2,1} \simeq \cdots \simeq e^{2,\ell} \fqH(\beta^{2,\ell}) e^{2,\ell}, $$
which implies
\begin{align} \label{Eq: isom 2}
e^{2,\ell} \fqH(\beta^{2,\ell}) e^{2,\ell} \simeq  \bR[x,y]/(x^2, y^2-axy).
\end{align}
Then, because of the dimension equality
$$  2 \dim e^{2,\ell} \fqH( \beta^{2,\ell} )  e^{2,\ell} = 8 = \dim e^{3,0} \fqH( \beta^{3,0} )  e^{3,0}, $$
the second monomorphism of $\eqref{Eq: fqH(3delta) embedding}$
\begin{align*} 
\bigoplus_{k=0}^1 \left( e^{2,\ell} \fqH(\beta^{2,\ell}) e^{2,\ell} \right) \otimes t^k  \hookrightarrow  e^{3,0}  \fqH(\beta^{3,0}) e^{3,0},
\end{align*}
sending $f \otimes t^k$ to $f x_{3\ell+3}^k e^{3,0} $ ($f \in e^{2,\ell} \fqH(\beta^{2,\ell}) e^{2,\ell} $) is an isomorphism of algebras.
Therefore, the assertion follows from $\eqref{Eq: isom 2}$ and the fact that $e^{3,0} \fqH( \beta^{3,0} )  e^{3,0}$ is graded.
\end{proof}

\begin{cor} \label{Cor: fqH(3delta) is wild}
The algebra $\fqH(3\delta)$ is of wild type.
\end{cor}
\begin{proof}
Let $A = e(\nu^{3,0}) \fqH(3\delta) e(\nu^{3,0})$. It is enough to show that $A$ is wild. By Proposition \ref{Prop: fqH(3delta)}, there is a surjective homomorphism
$$ A \twoheadrightarrow \bR[x,y,z] / (x^2, y^2, z^2, xy, yz, zx ). $$
Since the algebra $\bR[x,y,z] / (x^2, y^2, z^2, xy, yz, zx )$ is wild \cite[(1.2)]{Ringel75}, so is $A$.
\end{proof}

\subsection{ Representation type of $\fqH(\beta)$} In this subsection, we show our main theorem, which tells the representation type of $\fqH(\beta)$.

\begin{lemma}[{\cite[Prop.2.3]{EN02}}]
\label{Lem: reduction to critical rank}
Let $A$ and $B$ be finite dimensional $\bR$-algebras and suppose that
there exists a constant $C>0$ and functors
$$
F:\;A\text{\rm -mod} \rightarrow B\text{\rm -mod}, \quad
G:\;B\text{\rm -mod} \rightarrow A\text{\rm -mod}
$$
such that, for any $A$-module $M$,
\begin{itemize}
\item[(1)]
$M$ is a direct summand of $GF(M)$ as an $A$-module,
\item[(2)]
$\dim F(M)\le C\dim M$.
\end{itemize}
Then, if $A$ is wild, so is $B$.
\end{lemma}

\begin{Rmk}
In \cite[Prop.2.3]{EN02}, the authors use lengths of the modules to state the
result. But they use dimensions of the modules in the proof. Let us
quickly review their proof. We prove that if $B$ is tame then $A$ is
weakly tame. That is,
we show that there are finitely many $(\bR[T],A)$-bimodules such that,
for any indecomposable $A$-module $M$ with $\dim M\le d$,
$M$ is a direct summand of one of the bimodules tensored with
$\bR[T]/(T-\lambda)$, for some $\lambda\in \bR$. As $B$ is assumed to be
tame, we have finitely many $(\bR[T], B)$-bimodules such that every
indecomposable $B$-module with dimension at most $Cd$ is one of the
bimodules tensored with $\bR[T]/(T-\lambda)$. Thus, if $\dim F(M)$ is
bounded above by $Cd$, there are finitely many $(\bR[T], B)$-bimodules
such that $F(M)$ is a direct summand of one of them tensored with
$\bR[T]/(T-\lambda)$. By applying the functor $G$ to the bimodules, we
obtain finitely many $(\bR[T],A)$-bimodules with the required property.
The reason the authors use lengths is that we can replace $A$ or $B$ with
Morita equivalent algebras freely if we use lengths, and it is good for their
purposes.
\end{Rmk}

\begin{prop}\label{Prop: fqH(k delta)}
The algebras $\fqH(k\delta)$ $( k\ge3 )$ are wild.
\end{prop}
\begin{proof}
For $k\in \Z_{\ge0}$ and $0\le i\le\ell$, we have
$$
\langle h_i, \Lambda_0 - k\delta - \alpha_0- \cdots -\alpha_{i-1}  \rangle = \left\{
                                                                                \begin{array}{ll}
                                                                                  1 & \hbox{ if } i = 0, \dots, \ell-1, \\
                                                                                  2 & \hbox{ if } i = \ell.
                                                                                \end{array}
                                                                              \right.
$$
Thus, the functor $F_i : \fqH( k\delta + \alpha_0+ \cdots + \alpha_{i-1} )\text{-mod} \rightarrow
\fqH( k\delta + \alpha_0+ \cdots + \alpha_{i} )\text{-mod}$ satisfies the assumptions of Lemma \ref{Lem: reduction to critical rank}
by Proposition \ref{Prop: finite-dim} and Theorem \ref{Thm: Ei Fi}.
Hence, if $\fqH( k\delta + \alpha_0+ \cdots + \alpha_{i-1} )$ is wild, so is $\fqH( k\delta + \alpha_0+ \cdots + \alpha_{i} )$.
As $\fqH( 3\delta )$ is wild by Corollary \ref{Cor: fqH(3delta) is wild}, the assertion follows by induction.
\end{proof}

Recall that a weight $\mu$ with $V(\Lambda_0)_\mu \ne 0$ can be written as
$$\mu = \kappa - k \delta$$
 for some $\kappa \in \weyl \Lambda_0$ and $k\in \Z_{\ge0}$ and a weight $\mu$ of the above form always satisfies $V(\Lambda_0)_{\mu} \ne 0$
\cite[(12.6.1)]{Kac90}.
Note that the pair $(\kappa, k)$ is determined uniquely by $\mu$. Then, our main theorem, Theorem \ref{main theorem} below, follows from Proposition \ref{Prop: repn type}, Corollary \ref{Cor: fqH(2delta) repn type}
and Proposition \ref{Prop: fqH(k delta)}.

\begin{thm} \label{main theorem}
For $\kappa \in \weyl \Lambda_0$ and $k\in \Z_{\ge0}$, the finite quiver Hecke algebra $\fqH(\Lambda_0-\kappa+k\delta)$ of type $D_{\ell+1}^{(2)}$ ($\ell \ge 2$) is
\begin{enumerate}
\item simple if $k=0$,
\item of finite representation type but not semisimple if $k=1$,
\item of tame representation type if $k=2$,
\item of wild representation type if $k \ge 3$.
\end{enumerate}
\end{thm}

We give two remarks. Firstly, if $k=1$ then $\fqH(\Lambda_0-\kappa+\delta)$ is a matrix algebra over $\bR[x]/(x^2)$. To prove this,
observe that $\fqH(\delta)=\bR[x]/(x^2)$ is the Brauer tree algebra with one edge and no exceptional vertex. Then \cite[Thm.4.2]{Ri89} implies
that $\fqH(\Lambda_0-\kappa+\delta)$ is Morita equivalent to the Brauer tree algebra with one edge and no exceptional vertex, which is $\bR[x]/(x^2)$.
Thus, if we denote the dimension of the unique irreducible $\fqH(\Lambda_0-\kappa+\delta)$-module by $d$, we have
$\fqH(\Lambda_0-\kappa+\delta)\simeq \Mat(d,\bR[x]/(x^2))$. 

Secondly, if $k=2$ then $\fqH(\Lambda_0-\kappa+2\delta)$ is a symmetric algebra by \cite[Cor.5.3]{Ri91}.
If the results in \cite{AAC} are correct, then we may conclude that $\fqH(\Lambda_0-\kappa+2\delta)$ is a
symmetric special biserial algebra. Note that the main results in \cite{Po94} are incorrect,
as is mentioned in \cite[Example A.7]{AIP13}. 
Further, $\fqH(\Lambda_0-\kappa+2\delta)$ has the same
stable Auslander-Reiten quiver as $\fqH(2\delta)$, which is given in Proposition \ref{AR components}.

To summarize, if $\fqH(\beta)$ is not of wild representation type, then we know more than its representation type. For example, knowing the stable Auslander-Reiten
quiver implies that we may label indecomposable modules in some sense.


\bibliographystyle{amsplain}


\end{document}